\documentclass[12pt, reqno]{amsart} 
\usepackage{amssymb,amscd,amsfonts,amsbsy}
\usepackage{latexsym}
\usepackage{exscale}
\usepackage{amsmath,amsthm,amsfonts}
\usepackage{mathrsfs}
\usepackage{xcolor} 
\usepackage[colorlinks=true,linkcolor=blue,citecolor=red,urlcolor=red, backref=page]{hyperref} 
\usepackage{esint} 
\usepackage{stmaryrd}
\usepackage{pifont}

\usepackage[utf8]{inputenc}

\calclayout

\allowdisplaybreaks

\newtheorem{theorem}[equation]{Theorem}

\newtheorem{lemma}[equation]{Lemma}

\theoremstyle{definition}
\newtheorem{definition}[equation]{Definition}

\theoremstyle{remark}
\newtheorem{remark}[equation]{Remark}

\numberwithin{equation}{section} 

\def\diam{\operatorname{diam}}
\def\div{\operatorname{div}}
\def\dist{\operatorname{dist}}
\def\supp{\operatorname{supp}}

\def\BMO{\operatorname{BMO}}

\def\Lip{\operatorname{Lip}}
\def\loc{\operatorname{loc}}

\def\inter{\operatorname{int}}

\def\C{\mathbb{C}}
\def\N{\mathbb{N}}
\def\D{\mathbb{D}}

\def\F{\mathcal{F}}
\def\G{\mathcal{G}}
\def\H{\mathcal{H}}
\def\Z{\mathbb{Z}}
\def\R{\mathbb{R}}
\def\re{\mathbb{R}}

\def\Rn{\mathbb{R}^n}
\def\W{\mathcal{W}}

\def\w{\omega}

\def\pom{\partial\Omega}
\renewcommand{\emptyset}{\mbox{\textup{\O}}}

\begin{document}

\title[Elliptic measures and square function estimates]
{Elliptic measures and square function estimates on 1-sided chord-arc domains}

\author{Mingming Cao}
\address{Mingming Cao\\
Instituto de Ciencias Matem\'aticas CSIC-UAM-UC3M-UCM\\
Con\-se\-jo Superior de Investigaciones Cient{\'\i}ficas\\
C/ Nicol\'as Cabrera, 13-15\\
E-28049 Ma\-drid, Spain} \email{mingming.cao@icmat.es}

\author{Jos\'e Mar{\'\i}a Martell}
\address{Jos\'e Mar{\'\i}a Martell\\
Instituto de Ciencias Matem\'aticas CSIC-UAM-UC3M-UCM\\
Consejo Superior de Investigaciones Cient{\'\i}ficas\\
C/ Nicol\'as Cabrera, 13-15\\
E-28049 Ma\-drid, Spain} \email{chema.martell@icmat.es}

\author{Andrea Olivo}
\address{Andrea Olivo\\
Departamento de Matem\'atica\\ 
Facultad de Ciencias Exactas y Naturales\\ 
Universidad de  Buenos Aires and IMAS-CONICET\\
		Pabell\'on I (C1428EGA), Ciudad de Buenos Aires, Argentina} \email{aolivo@dm.uba.ar}

\thanks{The first author is supported by Spanish Ministry of Science and Innovation, through the Juan de la Cierva-Formación 2018, FJC2018-038526-I.
	 The first and second authors acknowledge financial support from the Spanish Ministry of Science and Innovation, through the ``Severo Ochoa Programme for Centres of Excellence in R\&D'' (CEX2019-000904-S) and  from the Spanish National Research Council, through the ``Ayuda extraordinaria a Centros de Excelencia Severo Ochoa'' (20205CEX001).  The second author also acknowledges that the research leading to these results has received funding from the European Research Council under the European Union's Seventh Framework Programme (FP7/2007-2013)/ ERC agreement no. 615112 HAPDEGMT. The second author was partially supported by the Spanish Ministry of Science and Innovation,  MTM PID2019-107914GB-I00. The last author was supported by PIP 112201501003553 (CONICET) and has received funding from the European Union’s Horizon 2020 research and innovation programme under the Marie Skłodowska-Curie grant agreement No 777822.   The last author would like to express her gratitude to the first two authors and the Instituto de Ciencias Matem\'aticas (ICMAT), for their support and hospitality.}

\date{\today}




\begin{abstract}
In nice environments, such as Lipschitz or chord-arc domains, it is well-known that the solvability of the  Dirichlet problem for an elliptic operator in $L^p$, for some finite $p$, is equivalent to the fact that the associated elliptic measure belongs to the Muckenhoupt class $A_\infty$. In turn, any of these conditions occurs if and only if the gradient of every bounded null solution satisfies  a Carleson measure estimate. This has been recently extended to much rougher settings such as those of 1-sided chord-arc domains, that is, sets which are quantitatively open and connected with a boundary which is Ahlfors-David regular. 
 In this paper, we work in the same environment and consider a qualitative analog of the latter equivalence showing that one can characterize the absolute continuity of the surface measure with respect to the elliptic measure  in terms of the finiteness almost everywhere of the truncated conical square function for any bounded null solution.  As a consequence of our main result particularized to the Laplace operator and some previous results, we show that the boundary of the domain is rectifiable if and only if the truncated conical square function is finite almost everywhere for any bounded harmonic function. Also, we obtain that for two given elliptic operators $L_1$ and $L_2$, the absolute continuity of the surface measure with respect to the elliptic measure of $L_1$ is equivalent to the same property for $L_2$ provided the disagreement  of the coefficients satisfy some quadratic estimate in truncated cones for almost everywhere vertex. Finally for the case on which $L_2$ is either the transpose of $L_1$ or its symmetric part we show the equivalence of the corresponding absolute continuity upon assuming that the antisymmetric part of the coefficients has some controlled oscillation in truncated cones for almost every vertex. 
\end{abstract}

\maketitle
\tableofcontents

\section{Introduction}\label{sec:intro}
A classical theorem of F. and M. Riesz \cite{RR} states that  
\begin{equation}\label{eq:Riesz}
\begin{array}{c}
\w \ll \mathcal{H}^1|_{\pom}\ll \w \text{ on } \partial \Omega \text{ for any simply connected } \\[0.2cm] 
\text{domain $\Omega \subset \R^2$ with a rectifiable boundary},  
\end{array}
\end{equation}
where $\w$ denotes the harmonic measure relative to the domain  $\Omega$. A quantitative version of this result was obtained later by Lavrentiev \cite{Lav} who showed that in a chord-arc domain in the plane, harmonic measure is quantitatively absolutely continuous with respect to  the arc-length measure, that is, harmonic measure is an $A_\infty$ weight with respect to surface measure. After these two fundamental results there  has been many authors seeking to find necessary and sufficient geometric criteria for the absolute continuity, or its quantitative version, of harmonic measure with respect to surface measure on the boundary of a domain in higher dimensions. In general, those can be divided into two categories: quantitative and qualitative.  

In the quantitative category it has been recently established that if $\Omega\subset\mathbb{R}^{n+1}$, $n\ge 2$, is a 1-sided CAD (chord-arc domain, cf. Definition \ref{defi:NTA}), then the following are equivalent: 
\begin{equation}\label{eq:UR}
\begin{split} 
&{\rm (a)}\quad \partial \Omega \text{ is uniformly rectifiable},  \\ 
&{\rm (b)}\quad \Omega \text{ satisfies the exterior corkscrew condition, hence it is a CAD},  \\
&{\rm (c)}\quad \w \in A_{\infty}(\sigma). 
\end{split}
\end{equation}
Here, $\sigma=\mathcal{H}^{n}|_{\pom}$ denotes the surface measure and $A_{\infty}(\sigma)$ is as mentioned above the scale-invariant version of absolute continuity.  The direction (a) implies (b) was shown by Azzam, Hofmann, Nyström, Toro, and the second named author of the present paper in \cite{AHMNT}. That (b) implies (c) was proved by David and Jerison in \cite{DJ}, and independently by Semmes in \cite{Se}.  Also, (a) implies (c) was proved by Hofmann and the second author of this paper in \cite{HM}. Both, jointly with Uriarte-Tuero \cite{HMU} also established that (c) implies (a). The equivalent statements in \eqref{eq:UR} reveal the close connection between the regularity of the boundary of a domain and the good behavior of harmonic measure with respect to surface measure. In addition,  \eqref{eq:UR} connects several known results, including the extension of \cite{RR} on Lipschitz domain \cite{D2}, $L^p_1$ domain \cite{JK1} and $\BMO_1$ domain \cite{JK2}. 

For divergence form elliptic operators $Lu=-\div(A \nabla u)$ with real variable coefficients,  that $(b)$ implies $(c)$ (with the elliptic measure $\omega_L$ in place of $\omega$) was proved by Kenig and Pipher in \cite{KP} under some Carleson measure estimate assumption for the matrix of coefficients $A$. The converse, that is, the fact that $(c)$ implies $(b)$ on a 1-sided CAD for the Kenig-Pipher class has been recently obtained by Hofmann, the second author of the present paper, Mayboroda, Toro, and Zhao in \cite{HMMTZ} (see also \cite{HMT1} for a previous result in a smaller class of operators). In another direction, it was shown in \cite{CHMT} that for any real (not necessarily non-symmetric) elliptic operator $L$, $\w_L \in A_{\infty}(\sigma)$ is equivalent to the so-called Carleson measure estimates, that is, every bounded weak null solution of $L$ satisfies Carleson measure estimates. 

On the other hand, the qualitative version of \eqref{eq:UR} has been also studied extensively. In contrast with \eqref{eq:Riesz},  some counterexamples have been presented to show how the absolute continuity of harmonic measure is indeed affected by the topology/geometry of the domain and its boundary. 
\begin{itemize}
	\item Example 1.  Lavrentiev constructed in \cite{Lav} a simply connected domain $\Omega \subset \R^2$ and a set $E \subset \partial \Omega$ such that $E$ has zero arclength, but $\w(E)>0$. 
	\vspace{0.2cm}
	\item Example 2. Bishop and Jones in \cite{BJ} found a uniformly rectifiable set $E$ on the plane and some subset of $E$ with zero arc-length which carries positive harmonic measure relative to the domain $\re^2\setminus E$.  
	\vspace{0.2cm}
	\item Example 3.  Wu proved in \cite{W} that there exists a topological ball $\Omega \subset \R^3$ and a set 
	$E \subset \partial \Omega$ lying on a 2-dimensional hyperplane so that Hausdorff dimension of $E$ is 1 
	(which implies $\sigma(E)=0$) but $\w(E)>0$.  
	\vspace{0.2cm}
	\item Example 4. In \cite{AMT}, Azzam, Mourgoglou and Tolsa obtained that for all $n \geq 2$,  there are a Reifenberg flat domain $\Omega \subset \R^{n+1}$ and a set $E \subset \partial \Omega$ such that $\w(E)>0=\sigma(E)$.
\end{itemize}
Compared with \eqref{eq:Riesz}, Examples 1 and 2 indicate that both the regularity of the boundary and the connectivity of the domain seem  to be necessary for absolute continuity to occur. However, Examples 3 and 4 say that $\w \ll \sigma$ fails in the presence of some connectivity assumption. Indeed, a quantitative form of path connectedness is contained in Example 4 since Reifenberg flat domains which are sufficiently flat are in fact NTA domains (cf. Definition \ref{defi:NTA}), see \cite[Theorem~3.1]{KT}. Taking into consideration these, it is natural to investigate what extra mild assumptions are necessary to obtain the absolute continuity of harmonic measure.  

It was shown by McMillan \cite[Theorem~2]{M} that for bounded simply connected domains $\Omega \subset \C$, $\w \ll \sigma \ll \w$ on the set of cone points. Later, Bishop and Jones \cite{BJ} obtained that for any simply connected domain $\Omega \subset \R^2$ and curve $\Gamma$ of finite length, $\w \ll \sigma$ on $\partial \Omega \cap \Gamma$. That result  refined the conclusions in \cite[p.~471]{O} and \cite[Theorem~3]{KW} where $\Gamma$ was a line and a quasi-smooth curve respectively. Beyond that, in a Wiener regular domain with large complement (cf. \cite[Definition~1.5]{AAM}), Akman, Azzam and Mourgoglou  \cite{AAM} gave a characterization of sets of absolute continuity in terms of the cone point condition and the rectifiable structure of elliptic measure. Let us point out that in all of the just mentioned results, the absolute continuity happens locally.  In the case of the whole  boundary, for every Lipschitz domain Dahlberg \cite{D1} proved that harmonic measure belongs to the reverse Hölder class with exponent $2$ with respect to surface measure, this, in turn, yields 
 $\w \ll \sigma \ll \w$ holds. This was extended to the setting of CAD domains in \cite{DJ, Se}.  For general NTA domains $\Omega \subset \mathbb{R}^{n+1}$, $n \geq 1$, Badger \cite{B} proved that $\sigma \ll \w$ if  the boundary $\partial \Omega$ has finite surface measure. When $\Omega$ is a 1-sided CAD, Akman, Badger, Hofmann and the second author established in \cite{ABHM} that $\partial \Omega$ is rectifiable if and only if $\sigma \ll \w$ on $\partial \Omega$, which is also equivalent to the fact that $\partial \Omega$ possesses exterior corkscrew points in a qualitative way and that $\partial \Omega$ can be covered $\sigma$-a.e. by a countable union of portions of boundaries of bounded chord-arc subdomains of $\Omega$. Based on a qualitative Carleson measure condition, they also got that the same conclusions hold for some class of elliptic operators with regular coefficients. The remarkable result in \cite{AHMMMTV} proved that, in any dimension and in the absence of any connectivity condition, any piece of the boundary with finite surface measure is rectifiable, provided surface measure is absolutely continuous with respect to harmonic measure on that piece. The converse was treated in \cite{ABoHM} by Akman, Bortz, Hofmann assuming that the boundary has locally finite surface measure and satisfies some weak lower Ahlfors regular condition.

Motivated by the previous work, the purpose of this article is to find characterizations of the absolute continuity of surface measure with respect to elliptic measure for real second order divergence form uniformly elliptic operators.  Our main goal is to establish the equivalence between the absolute continuity and the finiteness almost everywhere of the conical square function applied to any bounded weak solution. 
To set the stage let us give few definitions (see Section \ref{sec:pre} for more definitions and notation).  
The conical square function is defined as 
\begin{align*}
S_{\alpha}u(x) := \left(\iint_{\Gamma_{\alpha}(x)} 
|\nabla u(Y)|^2 \delta(Y)^{1-n} dY \right)^{\frac12}, \qquad x\in\pom,
\end{align*}
where $\delta(\cdot)=\dist(\cdot\,,\pom)$ and the cone $\Gamma_{\alpha}(x)$ with vertex at $x\in\pom$ and aperture $\alpha>0$ is given by 
\begin{align*}
\Gamma_{\alpha}(x) = \{Y \in \Omega: |Y-x|<(1+\alpha) \delta(Y) \}.
\end{align*}
Similarly,  we define the truncated square function $S_{\alpha}^r$ by integrating over the  truncated cone 
$\Gamma_{\alpha}^r(x) := \Gamma_{\alpha}(x) \cap B(x, r)$ for any $r>0$.  

Our main result is a qualitative analog of \cite{KKiPT} and \cite[Theorem 1.1]{CHMT}. More precisely, condition \eqref{list:wL} is a qualitative analog of $\omega_L\in A_\infty(\sigma)$ ---or equivalently $\sigma\in A_\infty(\omega_L)$--- while condition \eqref{list:Sr-L2}, or \eqref{list:Sr}, or \eqref{list:Srx} is a qualitative version of the so-called Carleson measure condition, which is in turn equivalent to some local scale-invariant $L^2$ estimate for the truncated conical square function.

\begin{theorem}\label{thm:abs-cont}
	Let $\Omega \subset \R^{n+1}$, $n\ge 2$,  be a $1$-sided CAD (cf. Definition \ref{defi:NTA}) and write  $\sigma:=\H^n |_{\partial \Omega}$. There exists $\alpha_0>0$ (depending only on the 
	$1$-sided CAD constants) such that for each fixed $\alpha \geq \alpha_0$ and for every real (not necessarily 
	symmetric) elliptic operator $Lu=-\div(A \nabla u)$ the following statements are equivalent: 
	\begin{list}{\textup{(\theenumi)}}{\usecounter{enumi}\leftmargin=1cm \labelwidth=1cm \itemsep=0.2cm 
			\topsep=.2cm \renewcommand{\theenumi}{\alph{enumi}}}
		\item\label{list:wL} $\sigma \ll \w_L$ on $\partial \Omega$. 
		
		\item\label{list:si-wL-si} $\partial \Omega=\bigcup_{N\geq 0} F_N$, where $\sigma(F_0)=0$ and for each 
		$N \geq 1$ there exists $C_N>1$ such that 
		\begin{align*}
		C_N^{-1} \sigma(F) \leq \w_L(F) \leq C_N \sigma(F), \quad \forall\,F \subset F_N.  
		\end{align*}
		
		\item\label{list:Sr-L2} $\partial \Omega=\bigcup_{N \geq 0} F_N$, where $\sigma(F_0)=0$, for each $N \geq 1$, 
		$F_N=\partial \Omega \cap \partial \Omega_N$ for some bounded $1$-sided CAD $\Omega_N \subset \Omega$, 
		and $S_{\alpha}^r u \in L^2(F_N, \sigma)$ for every weak solution $u \in W_{\loc}^{1,2}(\Omega) \cap L^{\infty}(\Omega)$ of $Lu=0$ in $\Omega$ and for all (or for some) $r>0$. 
		
		\item\label{list:Sr}  $S_{\alpha}^r u(x)<\infty$ for $\sigma$-a.e.~$x \in \partial \Omega$ for every weak solution $u \in W_{\loc}^{1,2}(\Omega)\cap L^{\infty}(\Omega)$ of  $Lu=0$ in $\Omega$ and 
		for all (or for some) $r>0$. 
		
		\item\label{list:Srx} For every weak solution $u \in W_{\loc}^{1,2}(\Omega)\cap L^{\infty}(\Omega)$ of 
		$Lu=0$ in $\Omega$ and for $\sigma$-a.e.~$x \in \partial \Omega$ there exists $r_x>0$ such that  
		$S_{\alpha}^{r_x} u(x)<\infty$.  
	\end{list}
\end{theorem}

\begin{remark}
	We would like to make the following observation regarding the parameter $\alpha$ in the previous statement. 
	Note first that if one of the conditions \eqref{list:Sr-L2}, \eqref{list:Sr}, or \eqref{list:Srx} holds for some $\alpha>0$,  
	then the same condition is automatically true for all $\alpha'\le \alpha$. Thus, \eqref{list:wL} or \eqref{list:si-wL-si} 
	implies \eqref{list:Sr-L2}, \eqref{list:Sr}, or \eqref{list:Srx}  holds for all $\alpha>0$. On the other hand, for the converse implications we need to make sure that $\alpha$ does not get too small to prevent having empty cones, in which case the corresponding assumption trivially holds. 
\end{remark}

When turning to the harmonic measure, we obtain the following connection between the rectifiability of the 
boundary, the absolute continuity of surface measure with respect to harmonic measure, and the square functions  estimates for harmonic functions.  

\begin{theorem}\label{thm:harmonic}
	Let $\Omega \subset \R^{n+1}$, $n\ge 2$, be a $1$-sided CAD and write  $\sigma:=\H^n |_{\partial \Omega}$. There exists $\alpha_0>0$ (depending only on the $1$-sided CAD constants) such that for each fixed $\alpha \geq \alpha_0$ if we write $\w$ to denote the harmonic measure for $\Omega$ then the following statements are equivalent: 
	\begin{list}{\textup{(\theenumi)}}{\usecounter{enumi}\leftmargin=1cm \labelwidth=1cm \itemsep=0.2cm \topsep=.2cm \renewcommand{\theenumi}{\alph{enumi}}}
		\item\label{list:rect} $\partial \Omega$ is rectifiable, that is, $\sigma$-almost all of $\partial\Omega$ can be covered by a countable union of $n$-dimensional (possibly rotated) Lipschitz graphs.
		\item\label{list:w} $\sigma \ll \w$ on $\partial \Omega$.
		\item\label{list:harmonic} $S_{\alpha}^r u(x)<\infty$ for $\sigma$-a.e.~$x \in \partial \Omega$ for every bounded harmonic function $u \in W_{\loc}^{1,2}(\Omega)$ and 
		for all (or for some) $r>0$. 
	\end{list}
\end{theorem}

The equivalence of \eqref{list:rect} and \eqref{list:w} was established in \cite{ABHM}, 
while Theorem \ref{thm:abs-cont} readily gives that \eqref{list:w} is equivalent to \eqref{list:harmonic}.

As an application of Theorem \ref{thm:abs-cont}, we can obtain some additional results. The first deals with perturbations (see \cite{AHMT, CHM, CHMT, D3, FJK, F, FKP, MPT1, MPT2}) and should be compared with its quantitative version in the 1-sided CAD setting \cite[Theorems 1.3]{CHMT}. We note that our next result provides also a quantitative version of the work by Fefferman in \cite{F} who showed that in the unit ball if the right hand side of \eqref{eq:rhoAA} is an essentially bounded function (rather than knowing that is finite almost everywhere) then one has $\omega_{L_0}  \in A_\infty(\sigma)$ if and only if $\omega_{L_1}  \in A_\infty(\sigma)$.

\begin{theorem}\label{thm:perturbation}
	Let $\Omega \subset \R^{n+1}$, $n\ge 2$,  be a $1$-sided CAD and write  $\sigma:=\H^n |_{\partial \Omega}$. There exists $\alpha_0>0$ (depending only on the 
	$1$-sided CAD constants) such that if the real (not necessarily symmetric) elliptic operators 
	$L_0 u = -\div(A_0 \nabla u)$ and $L_1 u = -\div(A_1 \nabla u)$ satisfy 
	for some $\alpha \ge\alpha_0$ and for some $r>0$
		\begin{align}\label{eq:rhoAA}
	\iint_{\Gamma_{\alpha}^{r}(x)} 
	\frac{\varrho(A_0, A_1)(X)^2}{\delta(X)^{n+1}} dX < \infty, 
	\quad \sigma \text{-a.e. } x \in \partial \Omega,
	\end{align}
	where
	\begin{equation*}
	\varrho(A_0, A_1)(X) :=\sup_{Y \in B(X, \delta(X)/2)} |A_0(Y) - A_1(Y)|, \quad X \in \Omega, 
	\end{equation*} 
		 then $\sigma \ll \w_{L_0}$ if and only if $\sigma \ll \w_{L_1}$. 
\end{theorem}

Our second application of Theorem \ref{thm:abs-cont} allows us to establish a connection between the absolute continuity properties of the elliptic measures of an operator, its adjoint and/or its symmetric part. Given $Lu=-\div(A \nabla u)$ a real (not necessarily 
symmetric) elliptic operator, we let $L^{\top}$ denote the transpose of $L$, and let $L^{\rm sym}=\frac{L+L^{\top}}{2}$ 
be the symmetric part of $L$. These are respectively the divergence form elliptic operators with associated matrices $A^\top$ (the transpose of $A$)  and $A^{\rm sym}=\frac{A+A^{\top}}{2}$. In this case, the following result is a qualitative version of \cite[Theorem 1.6]{CHMT}.

\begin{theorem}\label{thm:wL-wLT}
	Let $\Omega \subset \R^{n+1}$, $n\ge 2$, be a $1$-sided CAD and write  $\sigma:=\H^n |_{\partial \Omega}$.  There exists $\alpha_0>0$ (depending only on the 
	$1$-sided CAD constants) such that if $Lu=-\div(A \nabla u)$ is a real (not necessarily 
	symmetric) elliptic operator, and we assume that $(A-A^{\top}) \in \Lip_{\loc}(\Omega)$ and that for some $\alpha\ge \alpha_0$ and 
	for some $r>0$ one has 
	\begin{equation}\label{eq:divCAA}
	\iint_{\Gamma_{\alpha}^{r}(x)} 
	\left|\div_C(A-A^{\top})(X)\right|^2 \delta(X)^{1-n} dX<\infty, 
	\quad \sigma \text{-a.e. } x \in \partial \Omega, 
	\end{equation} 
	where 
	\begin{align*}
	\div_C(A-A^{\top})(X)=\bigg(\sum_{i=1}^{n+1} \partial_i(a_{i,j}-a_{j,i})(X) \bigg)_{1 \leq j \leq n+1}, 
	\quad X \in \Omega,  
	\end{align*} 
	then $\sigma \ll \w_L$ if and only if $\sigma \ll \w_{L^{\top}}$ if and only if $\sigma \ll \w_{L^{\rm sym}}$. 
\end{theorem}

\medskip

The paper is organized as follows. In Section \ref{sec:pre}, we present some preliminaries, definitions, and some background results that will be used throughout the paper. Section \ref{sec:proof-abs} is devoted to showing Theorem \ref{thm:abs-cont}. Finally, in Section \ref{sec:perturbation},  applying Theorem \ref{thm:abs-cont} $\eqref{list:wL} \Leftrightarrow \eqref{list:Sr}$, we obtain a more general  perturbation result about the absolute continuity of surface measure with respect to elliptic measure and then prove Theorems \ref{thm:perturbation} and \ref{thm:wL-wLT}.

\section{Preliminaries}\label{sec:pre}

\subsection{Notation and conventions}

\begin{list}{$\bullet$}{\leftmargin=0.4cm  \itemsep=0.2cm}
	
	\item  Our ambient space is $\re^{n+1}$, $n\ge 2$.
	
	\item We use the letters $c$, $C$ to denote harmless positive constants, not necessarily the same at each occurrence, which depend only on dimension and the constants appearing in the hypotheses of the theorems (which we refer to as the ``allowable parameters''). We shall also sometimes write $a\lesssim b$ and $a\approx b$ to mean, respectively, that $a\leq C b$ and $0<c\leq a/b\leq C$, where the constants $c$ and $C$ are as above, unless explicitly noted to the contrary. Moreover, if $c$ and $C$ depend on some given parameter $\eta$, which is somehow relevant, we write $a\lesssim_\eta b$ and $a\approx_\eta b$. At times, we shall designate by $M$ a particular constant whose value will remain unchanged throughout the proof of a given lemma or proposition, but which may have a different value during the proof of a different lemma or proposition.
	
	\item Given $E \subset \R^{n+1}$ we write $\diam(E)=\sup_{x, y \in E}|x-y|$ to denote its diameter.

	\item Given a domain (i.e., open and connected) $\Omega\subset\re^{n+1}$, we shall use lower case letters $x,y,z$, etc., to denote points on $\partial\Omega$, and capital letters $X,Y,Z$, etc., to denote generic points in $\re^{n+1}$ (especially those in $\Omega$).
	
	
	\item The open $(n+1)$-dimensional Euclidean ball of radius $r$ will be denoted $B(x,r)$ when the center $x$ lies on $\partial\Omega$, or $B(X,r)$ when the center $X\in\re^{n+1}\setminus \partial\Omega$. A ``surface ball'' is denoted $\Delta(x,r):=B(x,r)\cap \partial\Omega$, and unless otherwise specified it is implicitly assumed that $x\in\partial\Omega$. Also if $\partial\Omega$ is bounded, we typically assume that $0<r\lesssim\diam(\partial\Omega)$, so that $\Delta=\partial\Omega$ if $\diam(\partial\Omega)<r\lesssim\diam(\partial\Omega)$.

	\item Given a Euclidean ball $B$ or surface ball $\Delta$, its radius will be denoted $r_B$ or $r_\Delta$
	respectively.
	
	\item Given a Euclidean ball $B=B(X,r)$ or a surface ball $\Delta=\Delta(x,r)$, its concentric dilate by a factor of $\kappa>0$ will be denoted by $\kappa B=B(X,\kappa r)$ or $\kappa\Delta=\Delta(x,\kappa r)$.
	
	\item For $X\in\re^{n+1}$, we set $\delta(X):=\dist(X,\partial\Omega)$. 
	
	\item We let $\H^n$ denote the $n$-dimensional Hausdorff measure, and let $\sigma:=\H^n |_{\partial \Omega}$ denote the surface measure on $\partial \Omega$. 
	
%

	\item For a Borel set $A\subset\re^{n+1}$, we let $\inter(A)$ denote the interior of $A$, and $\overline{A}$ denote the closure of $A$. If $A\subset \partial\Omega$, $\inter(A)$ will denote the relative interior, i.e., the largest relatively open set in $\partial\Omega$ contained in $A$. Thus, for $A\subset \partial\Omega$, the boundary is then well defined by $\partial A:=\overline{A}\setminus\inter(A)$.
	
	\item For a Borel set $A\subset \partial\Omega$ with $0<\sigma(A)<\infty$, we write $\fint_{A}f\,d\sigma:=\sigma(A)^{-1}\int_A f\,d\sigma$.
	
	\item We shall use the letter $I$ (and sometimes $J$) to denote a closed $(n+1)$-dimensional Euclidean cube with 
	sides parallel to the coordinate axes, and we let $\ell(I)$ denote the side length of $I$. We use $Q$ to denote a dyadic ``cube'' on $\partial \Omega$. The latter exist, given that $\partial \Omega$ is ADR (see \cite{DS1}, \cite{Ch}, and enjoy certain properties which we enumerate in Lemma \ref{lem:dyadic} below).
	
\end{list}
\medskip

\subsection{Some definitions}

\begin{definition}[Ahlfors-David regular]  
We say that a closed set $E \subset \R^{n+1}$ is $n$-dimensional Ahlfors-David regular (or simply ADR) 
if there is some uniform constant $C\ge 1$ such that
\begin{equation*}
C^{-1} r^n \leq \H^n(E \cap B(x, r)) \leq C r^n, \qquad \forall\,x \in E, \ r \in (0, 2\,\diam(E)). 
\end{equation*} 
\end{definition} 
%

\begin{definition}[Corkscrew condition]\label{def:CKS} 
We say that an open set $\Omega \subset \R^{n+1}$ satisfies the Corkscrew condition if for some uniform 
constant $c \in (0, 1)$, and for every surface ball $\Delta:=\Delta(x, r)$ with $x \in \partial \Omega$ and 
$0< r < \diam(\partial \Omega)$, there is a ball $B(X_{\Delta}, cr) \subset B(x, r) \cap \Omega$. The point 
$X_{\Delta} \in \Omega$ is called a ``Corkscrew point'' relative to $\Delta$. We note that we may allow 
$r < C \diam(\partial \Omega)$ for any fixed $C$, simply by adjusting the constant $c$. 
\end{definition} 


\begin{definition}[Harnack Chain condition] 
We say that an open set $\Omega$ satisfies the Harnack Chain condition if there is a uniform constant $C$ such that for every 
$\rho>0$, $\Lambda \geq 1$, and every pair of points $X, X' \in \Omega$ with $\min\{\delta(X), \delta(X')\} \geq \rho$ 
and $|X-X'| < \Lambda \rho$, there is a chain of open balls $B_1, \ldots, B_N \subset \Omega$, $N \leq C(\Lambda)$, 
with $X \in B_1$, $X' \in B_N$, $B_k \cap B_{k+1} \neq \emptyset$, 
$C^{-1} \diam(B_k) \leq \dist(B_k, \partial \Omega) \leq C \diam(B_k)$. Such a sequence is called a ``Harnack Chain''. 
\end{definition} 

We remark that the Corkscrew condition is a quantitative, scale invariant version of openness,  and the 
Harnack Chain condition is a scale invariant version of path connectedness.

\begin{definition}[$1$-sided NTA domains, $1$-sided CAD, NTA domains, CAD] \label{defi:NTA}
We say that $\Omega$ is a $1$-sided NTA (non-tangentially accessible) domain if  $\Omega$ satisfies both the Corkscrew and 
Harnack Chain conditions. Furthermore, we say that $\Omega$ is an NTA domain if it is a $1$-sided NTA 
domain and if, in addition, $\R^{n+1} \setminus \overline{\Omega}$ satisfies the Corkscrew condition. 
If a $1$-sided NTA domain, or an NTA domain, has an ADR boundary, then it is called a 1-sided CAD (chord-arc domain) or a CAD, respectively.

\end{definition} 

\subsection{Dyadic grids and sawtooths}\label{section:dyadic}
We give a lemma concerning the existence of a ``dyadic grid'', which was proved in \cite{DS1, DS2, Ch}.  

\begin{lemma}\label{lem:dyadic}
Suppose that $E \subset \R^{n+1}$ is an $n$-dimensional $ADR$ set. Then there exist constants
$a_0>0$, $\gamma>0$, and $C_1 < 1$ depending only on $n$ and the $ADR$ constant 
such that, for each $k \in \Z$, there is a collection of Borel sets (cubes)
\begin{align*}
\D_k =\{Q_j^k \subset E: j \in \mathfrak{J}_k\}
\end{align*}
where $\mathfrak{J}_k$ denotes some (possibly finite) index set depending on $k$, satisfying: 
\begin{list}{$(\theenumi)$}{\usecounter{enumi}\leftmargin=1cm \labelwidth=1cm \itemsep=0.1cm \topsep=.2cm \renewcommand{\theenumi}{\alph{enumi}}}

\item $E=\bigcup_j Q_j^k$, for each $k \in \Z$. 
\item If $m \geq k$, then either $Q_i^m \subset Q_j^k$ or $Q_i^m \cap Q_j^k = \emptyset$. 
\item For each $(j, k)$ and each $m<k$, there is a unique $i$ such that $Q_j^k \subset Q_i^m$. 
\item $\diam(Q_j^k) \leq C_1 2^k$.  
\item Each $Q_j^k$ contains some surface ball $\Delta(x_j^k, a_0 2^{-k}):=B(x_j^k, a_0 2^{-k}) \cap E$.
\item $\H^n(\{x \in Q_j^k: \dist(x, E \backslash Q_j^k) \leq 2^{-k} a\}) \leq C_1 a^{\gamma} \H^n(Q_j^k)$ 
for all $k, j$ and $a \in (0, a_0)$.
\end{list}
\end{lemma}

A few remarks are in order concerning this lemma. 
\begin{enumerate}
\item[$\bullet$] In the setting of a general space of homogeneous type, this lemma has been proved by Christ \cite{Ch}, 
with the dyadic parameter $1/2$ replaced by some constant $\delta \in (0, 1)$. In fact, one may always take $\delta=1/2$ 
(cf. \cite[Proof of Proposition 2.12]{HMMM}). In the presence of the Ahlfors-David property, the result already appears in 
\cite{DS1, DS2}. 
\item[$\bullet$] For our purposes, we may ignore those $k \in \Z$ such that $2^{-k} \gtrsim \diam(E)$, in the case that 
the latter is finite.
\item[$\bullet$] We shall denote by $\D(E)$ the collection of all relevant $Q^k_j$, i.e., 
\begin{align*}
\D(E) := \bigcup_{k \in \Z} \D_k,
\end{align*}
where, if $\diam(E)$ is finite, the union runs over those $k$ such that $2^{-k} \lesssim \diam(E)$.  
\item[$\bullet$] For a dyadic cube $Q \in \D_k$, we shall set $\ell(Q)=2^{-k}$, and we shall refer to this quantity as the 
``length'' of $Q$. Evidently, $\ell(Q) \simeq \diam(Q)$. We set $k(Q)=k$ to be the dyadic generation to which $Q$ 
belongs if $Q \in \D_k$; thus, $\ell(Q)=2^{-k(Q)}$. 
\item[$\bullet$] Properties $(d)$ and $(e)$ imply that for each cube $Q \in \D_k$, there is a point $x_Q \in E$, a Euclidean ball 
$B(x_Q, r_Q)$ and a surface ball $\Delta(x_Q, r_Q) := B(x_Q, r_Q) \cap E$ such that $c\ell(Q) \leq r_Q \leq \ell(Q)$, for some uniform constant $c > 0$, and
\begin{align}\label{eq:Q-DQ} 
\Delta(x_Q, 2r_Q) \subset Q \subset \Delta(x_Q, Cr_Q), 
\end{align}
for some uniform constant $C>1$. We shall write 
\begin{align}\label{eq:BQ} 
B_Q := B(x_Q, r_Q),\quad \Delta_Q := \Delta(x_Q, r_Q), 
\quad \widetilde{\Delta}_Q := \Delta(x_Q, Cr_Q), 
\end{align}
and we shall refer to the point $x_Q$ as the ``center'' of $Q$. 
\item[$\bullet$] Let $\Omega \subset \R^{n+1}$ be an open set satisfying the corkscrew condition and such that 
$\partial \Omega$ is ADR. Given $Q \in \D(\partial \Omega)$, we define the ``corkscrew point relative to $Q$'' as 
$X_Q:=X_{\Delta_Q}$. We note that 
\[ \delta(X_Q) \simeq \dist(X_Q, Q) \simeq \diam(Q). \] 
\end{enumerate}

We next introduce the notation of ``Carleson region'' and ``discretized sawtooth'' from \cite[Section 3]{HM}. 
Given a dyadic cube $Q \in \D(E)$, the ``discretized Carleson region'' $\D_Q$ relative to $Q$ is defined by 
\[ \D_Q := \{Q' \in \D(E) : Q' \subset Q\}. \] 
Let $\F=\{Q_j\} \subset \D(E)$ be a pairwise family of disjoint cubes. The ``global discretized sawtooth'' relative to $\F$ 
is the collection of cubes $Q \in \D(E)$ that are not contained in any $Q_j \in \F$, that is, 
\begin{equation*}
\D_{\F} :=\D(E) \setminus \bigcup_{Q_j \in \F} \D_{Q_j}.  
\end{equation*}
For a given cube $Q \in \D(E)$, we define the ``local discretized sawtooth'' relative to $\F$ is the collection of cubes in 
$\D_Q$ that are not contained in any $Q_j \in \F$ of, equivalently, 
\begin{equation*}
\D_{\F, Q} :=\D_Q \setminus \bigcup_{Q_j \in \F} \D_{Q_j}=\D_{\F} \cap \D_Q. 
\end{equation*} 

We also introduce the ``geometric'' Carleson regions and sawtooths. In the sequel, $\Omega \subset \R^{n+1}$,  
$n \geq 2$, is a $1$-sided CAD. Given $Q \in \D:=\D(\partial \Omega)$ we want to define some associated regions 
which inherit the good properties of $\Omega$. Let $\W=\W(\Omega)$ denote a collection of (closed) dyadic Whitney cubes 
of $\Omega$, so that the cubes in $\W$ form a covering of $\Omega$ with non-overlapping interiors, which satisfy 
\begin{equation*}
4 \diam(I) \leq \dist(4I, \partial \Omega) \leq \dist(I, \partial \Omega) \leq 40 \diam(I), \quad \forall\,I \in \W, 
\end{equation*}
and also 
\begin{equation*}
(1/4) \diam(I_1) \leq \diam(I_2) \leq 4 \diam(I_1), \quad\text{whenever $I_1$ and $I_2$ touch}.
\end{equation*}
Let $X(I)$ be the center of $I$ and $\ell(I)$ denote the sidelength of $I$. 

Given $0<\lambda<1$ and $I \in \W$, we write $I^*=(1+\lambda)I$ for the ``fattening'' of $I$. 
By taking $\lambda$ small enough, we can arrange matters, so that for any $I, J \in \W$, 
\begin{equation*}
\begin{aligned} 
\dist(I^*, J^*) & \simeq \dist(I, J), \\ 
\operatorname{int}(I^*) \cap \operatorname{int}(J^*) \neq \emptyset & 
\Longleftrightarrow \partial I \cap \partial J \neq \emptyset.  
\end{aligned}
\end{equation*}
(The fattening thus ensures overlap of $I^*$ and $J^*$ for any pair $I, J \in \W$ whose boundaries touch, so that the 
Harnack chain property then holds locally, with constants depending upon $\lambda$, in $I^* \cap J^*$.) By choosing 
$\lambda$ sufficiently small, say $0<\lambda<\lambda_0$, we may also suppose that there is a $\tau \in (1/2, 1)$ such 
that for distinct $I, J \in \W$, we have that $\tau J \cap I^{*}=\emptyset$.  
In what follows we will need to work with the dilations $I^{**}=(1+2\lambda)I$ or $I^{***}=(1+4\lambda)I$, and in 
order to ensure that the same properties hold we further assume that $0<\lambda<\lambda_0/4$.

Given $\vartheta\in\mathbb{N}$,  for every cube $Q \in \D$ we set  
\begin{equation}\label{eq:WQ}
\W_Q^\vartheta :=\left\{I \in \W: 2^{-\vartheta}\ell(Q) \leq \ell(I) \leq 2^\vartheta\ell(Q), \text { and } \dist(I, Q) \leq 2^\vartheta \ell(Q) \right\}.
\end{equation}
We will choose $\vartheta\ge \vartheta_0$,  with $\vartheta_0$ large enough depending on the constants of the corkscrew condition (cf. Definition  \ref{def:CKS}) and in the dyadic cube construction
(cf. Lemma \ref{lem:dyadic}), so that $X_Q \in I$ for some $I \in \W_Q^\vartheta$, and for each dyadic child $Q^j$ of $Q$, the respective corkscrew points 
$X_{Q^j} \in I^j$ for some $I^j  \in \W_Q^\vartheta$. Moreover, we may always find an $I \in \W_Q^\vartheta$ with the slightly more precise 
property that $\ell(Q)/2 \leq \ell(I) \leq \ell(Q)$ and 
\begin{equation*}
\W_{Q_1}^\vartheta \cap \W_{Q_2}^\vartheta \neq \emptyset, \quad \text { whenever } 
1 \leq \frac{\ell(Q_2)}{\ell(Q_1)} \leq 2, \text { and } \dist(Q_1, Q_2) \leq 1000 \ell(Q_2). 
\end{equation*}

 For each $I \in \W_Q^\vartheta$, we form a Harnack chain from the center $X(I)$ to the corkscrew point $X_Q$ and call it $H(I)$. 
We now let $\W_{Q}^{\vartheta, *}$ denote the collection of all Whitney cubes which meet at least one ball in the Harnack chain $H(I)$ 
with $I \in \W_Q$, that is, 
\begin{equation*}
\W_{Q}^{\vartheta, *}:=\{J \in \W: \text{ there exists } I \in \W_Q \text{ such that } H(I) \cap J \neq \emptyset\}. 
	\end{equation*}
We also define 
\begin{equation*}
U_{Q}^\vartheta :=\bigcup_{I \in \W_{Q}^{\vartheta, *}}(1+\lambda) I=: \bigcup_{I \in \W_{Q}^{\vartheta, *}} I^{*}. 
\end{equation*}
By construction, we then have that 
\begin{equation*}
\W_{Q}^\vartheta  \subset \W_{Q}^{\vartheta, *} \subset \W \quad \text{and}\quad X_Q \in U_Q^{\vartheta}, \quad X_{Q^{j}} \in U_{Q}^{\vartheta}, 
\end{equation*} 
for each child $Q^j$ of $Q$. It is also clear that there is a uniform constant $k^*$ (depending only on the 
$1$-sided CAD constants and $\vartheta$) such that 
\begin{align*}
2^{-k^*} \ell(Q) \leq \ell(I) \leq 2^{k^*}\ell(Q), &\quad \forall\,I \in \W_{Q}^{\vartheta, *},  
\\
X(I) \rightarrow_{U_Q^\vartheta} X_Q, &\quad \forall\,I \in \W_{Q}^{\vartheta, *},
\\
\dist(I, Q) \leq 2^{k^*} \ell(Q), &\quad \forall\,I \in \W_{Q}^{\vartheta, *}. 
\end{align*} 
Here, $X(I) \to_{U_Q^\vartheta} X_Q$ means that the interior of $U_Q^\vartheta$ contains all balls in Harnack Chain (in $\Omega$) connecting $X(I)$ to $X_Q$, and moreover, for any point $Z$ contained in any ball in the Harnack Chain, we have 
$\dist(Z, \partial \Omega) \simeq \dist(Z, \Omega \setminus U_Q^\vartheta)$ with uniform control of implicit constants. The 
constant $k^*$ and the implicit constants in the condition $X(I) \to_{U_Q^\vartheta} X_Q$, depend on at most allowable 
parameter, on $\lambda$, and on $\vartheta$. Moreover, given $I \in \W$ we have that $I \in \W^{\vartheta,*}_{Q_I}$, where $Q_I \in \D(\partial \Omega)$ satisfies $\ell(Q_I)=\ell(I)$, and contains any fixed $\widehat{y} \in \partial \Omega$ such that 
$\dist(I, \partial \Omega)=\dist(I, \widehat{y})$. The reader is referred to \cite{HM} for full details. We note however that in that reference the parameter $\vartheta$ is fixed. Here we need to allow $\vartheta$ to depend on the aperture of the cones and hence it is convenient to include the superindex $\vartheta$.

For a given $Q \in \D$, the ``Carleson box'' relative to $Q$ is defined by  
\begin{equation*}
T_{Q}^\vartheta :=\operatorname{int}\bigg(\bigcup_{Q' \in \D_Q} U_{Q'}^\vartheta\bigg). 
\end{equation*}  
For a given family $\F=\{Q_j\}$ of pairwise disjoint cubes and a given $Q \in \D(\partial \Omega)$, we define the 
``local sawtooth region'' relative to $\F$ by 
\begin{equation*}
\Omega_{\F, Q}^\vartheta :=\inter\bigg(\bigcup_{Q' \in \D_{\F, Q}} U_{Q'}^\vartheta\bigg)
=\inter\bigg(\bigcup_{I \in \W_{\F, Q}^\vartheta} I^* \bigg), 
\end{equation*}
where $\W_{\F, Q}^\vartheta:=\bigcup_{Q' \in \D_{\F, Q}} \W^{\vartheta,*}_Q$. Analogously, we can slightly fatten the Whitney boxes and use 
$I^{**}$ to define new fattened Whitney regions and sawtooth domains. More precisely, for every $Q \in \D(\partial \Omega)$,
\begin{equation*}
T^{\vartheta,*}_Q :=\inter\bigg(\bigcup_{Q' \in \D_Q} U_{Q'}^\vartheta\bigg), \quad 
\Omega^{\vartheta,*}_{\F, Q} :=\inter\bigg(\bigcup_{Q' \in \D_{\F, Q}} U^{\vartheta,*}_{Q'}\bigg), \quad 
U^{\vartheta,*}_Q :=\bigcup_{I \in \W^{\vartheta,*}_Q} I^{**}.  
\end{equation*}
Similarly, we can define $T^{\vartheta,**}_Q$, $\Omega^{\vartheta,**}_{\F, Q}$ and $U^{\vartheta,**}_Q$ by using $I^{***}$ in place of $I^{**}$. 
For later use, we recall that \cite[Proposition 6.1]{HM}: 
\begin{equation}\label{eq:boundary}
Q \backslash \bigg(\bigcup_{Q_j \in \F} Q_j\bigg) 
\subset \partial \Omega \cap \partial \Omega_{\F, Q}^\vartheta 
\subset \overline{Q} \backslash \bigg(\bigcup_{Q_j \in \F} \inter(Q_j)\bigg).  
\end{equation}

Following \cite{HM}, one can easily see that there exist constants $0<\kappa_1<1$ and 
$\kappa_0 \geq \max\{2C, 4/c\}$ (with $C$ the constant in \eqref{eq:Q-DQ}, and $c$ such that 
$c \ell(Q) \leq r_Q$), depending only on the allowable parameters and on $\vartheta$, so that
\begin{align}\label{eq:kappa}
\kappa_1B_Q \cap \Omega \subset T_Q^\vartheta \subset T^{\vartheta,*}_Q \subset T^{\vartheta,**}_Q \subset \overline{T^{\vartheta,**}_Q} 
\subset \kappa_0 B_Q \cap \overline{\Omega} =: \frac12 B^*_Q \cap \overline{\Omega}, 
\end{align}
where $B_Q$ is defined as in \eqref{eq:BQ}. 

\subsection{PDE estimates} 
Now we recall several facts concerning the elliptic measures and the Green functions. For our first results 
we will only assume that $\Omega \subset \R^{n+1}$, $n \geq 2$, is an open set, not necessarily connected, 
with $\partial \Omega$ being ADR. Later we will focus on the case where $\Omega$ is a 
$1$-sided CAD. 

Let $Lu = - \div(A \nabla u)$ be a variable coefficient second order divergence form operator with 
$A(X)=(a_{i, j}(X))_{i,j=1}^{n+1}$ being a real (not necessarily symmetric) matrix with 
$a_{i, j} \in L^{\infty}(\Omega)$ for $1 \leq i, j \leq n+1$, and $A$ uniformly elliptic, that is, there 
exists $\Lambda \geq 1$ such that 
\begin{align*}
\Lambda^{-1} |\xi|^2 \leq A(X) \xi \cdot \xi,\quad |A(X) \xi \cdot \eta| \leq \Lambda |\xi| |\eta|,
\quad \forall\,\xi, \eta \in \R^{n+1} \text{ and  a.e.~}X \in \Omega.
\end{align*}

In what follows we will only be working with this kind of operators, we will refer to them as ``elliptic operators'' 
for the sake of simplicity. We write $L^{\top}$ to denote the transpose of $L$, or, in other words, 
$L^{\top}u = -\div(A^{\top}\nabla u)$ with $A^{\top}$ being the transpose matrix of $A$. 

We say that a function $u \in W_{\loc}^{1,2}(\Omega)$ is a weak solution of $Lu=0$ in $\Omega$, or that
$Lu=0$ in the weak sense, if 
\begin{align*}
\iint_{\Omega} A(X) \nabla u(X) \cdot \nabla \phi(X)=0, \quad \forall\,\phi \in C_c^{\infty}(\Omega). 
\end{align*}
Here and elsewhere $C_c^{\infty}(\Omega)$ stands for the set of compactly supported smooth functions with all derivatives of all orders being continuous.

Associated with the operators $L$ and $L^{\top}$,  one can respectively construct the elliptic measures 
$\{\w_L^X\}_{X \in \Omega}$ and $\{\w_{L^\top}^X\}_{X \in \Omega}$, and the Green functions 
$G_L$ and $G_{L^{\top}}$ (see \cite{HMT2} for full details). We next present some definitions and properties 
that will be used throughout this paper.

The following lemmas can be found in \cite{HMT2}.

\begin{lemma}
Suppose that $\Omega \subset  \R^{n+1}$, $n\ge 2$, is an open set such that $\partial \Omega$ is 
ADR. Given an elliptic operator $L$, there exist $C>1$ (depending only on dimension 
and on the ellipticity of $L$) and $c_{\theta}>0$ (depending on the above parameters and on 
$\theta \in (0, 1)$) such that $G_L$, the Green function associated with $L$, satisfies 
\begin{align}
G_L(X, Y) &\leq C|X-Y|^{1-n}; \\ 
c_{\theta} |X-Y|^{1-n} \leq G_L(X, Y), \quad &\text{if } |X-Y| \leq \theta \delta(X), \theta \in(0,1); \\
G_L(\cdot, Y) \in C(\overline{\Omega} \setminus \{Y\}) &\text{ and } 
G_L(\cdot, Y)|_{\partial \Omega} \equiv 0, \forall\,Y \in \Omega;  \\ 
G_L(X, Y) \geq 0, &\quad \forall\,X, Y \in \Omega, X \neq Y;  \\ 
G_L(X, Y)=G_{L^{\top}}(Y, X), &\quad \forall\,X, Y \in \Omega, X \neq Y;  
\end{align}
Moreover, $G_L(\cdot, Y) \in W^{1,2}_{\loc}(\Omega \setminus \{Y\})$ for every $Y \in \Omega$, 
and satisfies $LG_L(\cdot, Y)=\delta_Y$ in the weak sense in $\Omega$, that is,  
\begin{equation}
\iint_{\Omega} A(X) \nabla_{X} G_L(X, Y) \cdot \nabla \Phi(X) dX = \Phi(Y), \quad \forall\,\Phi \in C_c^{\infty}(\Omega). 
\end{equation} 

Finally, the following Riesz formula holds 
\begin{equation}
\iint_{\Omega} A^{\top}(Y) \nabla_{Y} G_{L^{\top}}(Y, X) \cdot \nabla \Phi(Y) dY 
= \Phi(X) - \int_{\partial \Omega} \Phi d\w_L^X, 
\end{equation} 
for a.e.~$X \in \Omega$ and for every $\Phi \in C_c^{\infty} (\R^{n+1})$.  
\end{lemma}

\begin{lemma}\label{lemma:proppde}
Suppose that $\Omega \subset  \R^{n+1}$, $n\ge 2$, is a $1$-sided CAD. Let $L$ be an elliptic operator. There exists a constant $C$ 
(depending only on the dimension, the $1$-sided CAD constants and the ellipticity of $L$) such that for 
every ball $B_0:=B(x_0, r_0)$ with $x_0 \in \partial \Omega$ and $0<r_0<\diam(\partial \Omega)$, and $\Delta_0=B_0\cap\pom$ we 
have the following properties: 

\begin{list}{\textup{(\theenumi)}}{\usecounter{enumi}\leftmargin=1cm \labelwidth=1cm \itemsep=0.2cm \topsep=.2cm \renewcommand{\theenumi}{\alph{enumi}}}

\item There holds  
\begin{align}\label{Bourgain}
\w_L^Y(\Delta_0) \geq 1/C, \quad \forall\,Y \in \Omega \cap B(x_0, C^{-1}r_0). 
\end{align}

\item  If $B=B(x, r)$ with $x \in \partial \Omega$ is such that $2B \subset B_0$, then for any 
$X \in \Omega \setminus B_0$, 
\begin{align}\label{eq:comparison} 
C^{-1} \w_L^{X}(\Delta) \leq r^{n-1} G_L(X, X_{\Delta}) \leq C \w_L^{X}(\Delta).  
\end{align} 
\item[$(b)$] If $X \in \Omega \backslash 4B_0$, then we have 
\begin{align}\label{eq:doubling} 
\w_L^X(2\Delta_0) \leq C \w_L^X(\Delta_0).  
\end{align}
\end{list}
\end{lemma} 

\section{Proof of Theorem \ref{thm:abs-cont}}\label{sec:proof-abs}

The goal of this section is to prove Theorem \ref{thm:abs-cont}. We start with the following observation which will be used throughout the paper:

\begin{remark}\label{remark:truncations}
For every $\alpha>0$, $0<r<r'$, and $\varpi\in \R$, if $F \subset\pom$ is a bounded set  and $v\in L^2_{\loc}(\Omega)$, then
\begin{equation}\label{est:two-trunc}
\sup_{x\in F}\iint_{\Gamma_{\alpha}^{r'}(x) \setminus \Gamma_{\alpha}^{r}(x)} 
|v(Y)|^2 \delta(Y)^{\varpi} dY 
<\infty.  
\end{equation}
To see this we first note that  since $F$ is bounded we can find $R$ large enough so that $F\subset B(0,R)$. Then, if $x\in F$ one readily sees that 
\begin{align*}
\Gamma_{\alpha}^{r'}(x) \backslash \Gamma_{\alpha}^{r}(x) 
\subset \overline{B(0, r'+R)} \cap 
\Big\{Y \in \Omega: \frac{r}{1+\alpha} \leq \delta(Y) \leq r' \Big\} =:K. 
\end{align*}
Note that $K \subset \Omega$ is a compact set.  Then, since  $v \in L^2_{\loc}(\Omega)$, we conclude that 
\begin{align}\label{eq:Gr-Gr0}
\sup_{x\in F}\iint_{\Gamma_{\alpha}^{r'}(x) \setminus \Gamma_{\alpha}^{r}(x)} 
|v(Y)|^2 \delta(Y)^{\varpi} dY 
\leq \max\left\{r',\frac{1+\alpha}{r}\right\}^{|\varpi|} \iint_{K}  |v(Y)|^2 dY<\infty. 
\end{align}
\end{remark}

We can now proceed to prove Theorem \ref{thm:abs-cont}.  We first note that it is immediate to see that   $\eqref{list:si-wL-si}\Longrightarrow\eqref{list:wL}$, $\eqref{list:Sr-L2}\Longrightarrow\eqref{list:Sr}$, and $\eqref{list:Sr}\Longrightarrow\eqref{list:Srx}$. Moreover, \eqref{est:two-trunc} yields easily $\eqref{list:Srx}\Longrightarrow\eqref{list:Sr}$.  
Thus, it suffices to prove the following implications: 
\begin{align*}
\eqref{list:wL} \Longrightarrow \eqref{list:Sr-L2}, \qquad 
\eqref{list:wL} \Longrightarrow \eqref{list:si-wL-si}, \qquad\text{and}\qquad 
\eqref{list:Sr} \Longrightarrow \eqref{list:wL}. 
\end{align*}

\subsection{Proof of \texorpdfstring{$\eqref{list:wL} \Longrightarrow \eqref{list:Sr-L2}$}{(a) implies (c)}}\label{sec:a-c}
Assume that $\sigma \ll \w_L$. 
Fix and arbitrary $Q_0 \in \D_{k_0}$ where $k_0 \in \Z$ is taken so that $2^{-k_0}=\ell(Q_0) < \diam(\partial \Omega)/M_0$, 
where $M_0$ is large enough and will be chosen later. From the construction of $T_{Q_0}^{\vartheta}$ one can easily see that 
$T_{Q_0}^{\vartheta} \subset \frac12B_{Q_0}^*:=\kappa_0 B_{Q_0}$, see \eqref{eq:kappa}. Let $X_0:=X_{M_0 \Delta_{Q_0}}$ be an interior corkscrew point relative to $M_0 \Delta_{Q_0}$ so that $X_0 \notin 4B_{Q_0}^*$ provided that $M_0$ is taken large enough depending on the allowable parameters. Since $\partial \Omega$ is ADR, \eqref{Bourgain} and Harnack's inequality give that $\w_L^{X_0}(Q_0) \geq C_0^{-1}$, where $C_0>1$ depends on 1-sided CAD constants and $M_0$. We now normalize the elliptic measure and the Green function as follows 
\begin{align}\label{eq:normalize}
\w:= C_0 \sigma(Q_0) \w_L^{X_0} \quad\text{ and }\quad 
\G(\cdot) := C_0 \sigma(Q_0) G_L(X_0, \cdot). 
\end{align} 
The hypothesis $\sigma \ll \w_L$ implies that $\sigma \ll \w$. Note that $1 \leq \frac{\w(Q_0)}{\sigma(Q_0)} \leq C_0$. 
Let $N > C_0$ and let $\F_N^+ :=\{Q_j\} \subset \D_{Q_0} \backslash \{Q_0\}$, respectively,
$\F_N^- :=\{Q_j\} \subset \D_{Q_0} \backslash \{Q_0\}$,  be the collection of descendants of $Q_0$ 
which are maximal (and therefore pairwise disjoint) with respect to the property that 
\begin{align}\label{eq:stopping}
\frac{\w(Q_j)}{\sigma(Q_j)} < \frac{1}{N}, \quad\text{ respectively }\quad 
\frac{\w(Q_j)}{\sigma(Q_j)} >N. 
\end{align}
Write $\F_N=\F_N^+\cup\F_N^-$ and note that $\F_N^+\cap\F_N^-=\emptyset$. By maximality, one has 
\begin{align}\label{eq:NN}
\frac{1}{N}\leq \frac{\w(Q)}{\sigma(Q)} \leq N, \quad \forall\,Q \in \D_{\F_N, Q_0}. 
\end{align}
Write for every $N>C_0$, 
\begin{align}\label{eq:E0N-EN}
E_N^\pm := \bigcup_{Q \in \F_N^\pm} Q,
\qquad
E_N^0=E_N^+\cup E_N^-, \qquad
E_N := Q_0\setminus E_N^0,
\end{align}
and
\begin{align}\label{eq:Q-decom}
Q_0=
\bigg(\bigcap_{N>C_0} E_N^0\bigg)\cup \bigg(\bigcup_{N>C_0} E_N \bigg)
=: E_0\cup \bigg(\bigcup_{N>C_0} E_N \bigg)
. 
\end{align}
We claim that for every $N>C_0$
\begin{align}\label{qafferr}
E_N^+=\{x\in Q_0: M_{Q_0, \w}^{\text{d}} \sigma(x)>N \}\quad\text{and}\quad E_N^-=\{x\in Q_0: M_{Q_0, \sigma}^{\text{d}} \w(x)>N \},
\end{align}
where given two non-negative Borel measures $\mu$ and $\nu$ we set
\begin{align*}
M_{Q_0, \mu}^{\text{d}}\nu (x) := \sup_{x \in Q \in \D_{Q_0}} 
\frac{\nu(Q)}{\mu(Q)}. 
\end{align*}
To see the first equality in \eqref{qafferr}, let $x\in E_N^+$. Then, there exists $Q_j\in \F_N^+\subset\D_{Q_0}$ so that $Q_j\ni x$. Thus, by \eqref{eq:stopping}
\[
M_{Q_0, \w}^{\text{d}}\sigma (x)
\ge
\frac{\sigma(Q_j)}{\w(Q_j)}>N.
\]
On the other hand, if $M_{Q_0, \w}^{\text{d}}\sigma (x)>N$, there exists $Q\in\D_{Q_0}$ so that $\sigma(Q)/\w(Q)>N$. By the maximality of $\F_N^+$ we therefore conclude that $Q\subset Q_j$ for some $Q_j\in\F_N^+$. Hence, $x\in E_N^+$ as desired. This completes the proof of the first equality in \eqref{qafferr} and the second one follows using the same argument interchanging the roles of $\w$ and $\sigma$.

Once we have shown  \eqref{qafferr}, we clearly see that $\{E_N^+\}_N$, $\{E_N^-\}_N$, and $\{E_N^0\}_N$ are decreasing sequences of sets.  This, together with the fact that $\omega(E_N^\pm)\le \omega(Q_0)<\infty$ and $\sigma(E_N^\pm)\le \sigma(Q_0)<\infty$, implies that
\begin{equation}\label{wrqfawfvrw}
\omega\bigg(\bigcap_{N>C_0} E_N^\pm\bigg)=\lim_{N\to\infty} \omega(E_N^\pm),
\qquad
\sigma\bigg(\bigcap_{N>C_0} E_N^\pm\bigg)=\lim_{N\to\infty} \sigma(E_N^\pm).
\end{equation}

Our next goal is to show 
that $\sigma(E_0)=0$. To see this we note that by \eqref{eq:stopping}
\[
\omega(E_N^+)
=
\sum_{Q\in \F_N^+} \omega(Q)
<\frac1N\sum_{Q\in \F_N^+} \sigma(Q)
=\frac1N\sigma(E_N^+)
\le 
\frac1N\sigma(Q_0)
\]
and, by \eqref{wrqfawfvrw}
\[
\omega\bigg(\bigcap_{N>C_0} E_N^+\bigg)=\lim_{N\to\infty} \omega(E_N^+)=0.
\]
Use this, the fact that $\sigma \ll \w$, and \eqref{wrqfawfvrw} to derive
\begin{equation}\label{54wt5gt}
0=\sigma\bigg(\bigcap_{N>C_0} E_N^+\bigg)=\lim_{N\to\infty} \sigma(E_N^+).
\end{equation}
On the other hand, \eqref{eq:stopping} yields
\[
\sigma(E_N^-)
=
\sum_{Q\in \F_N^-} \sigma(Q)
<\frac1N\sum_{Q\in \F_N^-} \omega(Q)
=\frac1N\omega(E_N^-)
\le 
\frac1N\omega(Q_0),
\]
and \eqref{wrqfawfvrw} implies 
\begin{equation}\label{54wt5gt:alt}
\sigma\bigg(\bigcap_{N>C_0} E_N^-\bigg)=\lim_{N\to\infty} \sigma(E_N^-)=0.
\end{equation}
All these, together with \eqref{wrqfawfvrw} and the fact that $\{E_N^0\}_N$ is a decreasing sequence of sets with $\sigma(E_N^0)\le\sigma(Q_0)<\infty$, give
\begin{equation}\label{34fravcrv}
\sigma(E_0)
=
\lim_{N\to\infty}
\sigma(E_N^0)
\le 
\lim_{N\to\infty}
\sigma(E_N^+)+\lim_{N\to\infty}\sigma(E_N^-)
=0,
\end{equation}
hence $\sigma(E_0)=0$.

Next, by \eqref{eq:boundary} and \cite[Proposition 6.3]{HM}, we have 
\begin{align}\label{eq:EN-FN}
E_N \subset F_N:=\partial \Omega \cap \partial \Omega_{\F_N, Q_0}^{\vartheta}  
\qquad\text{and}\qquad \sigma(F_N \backslash E_N)=0.
\end{align}  
Note that \cite[Lemma 3.61]{HM} yields that $\Omega_{\F_N, Q_0}^\vartheta$ is a bounded $1$-sided CAD for any $\vartheta\ge \vartheta_0$. 
Now we are going to bound the square function in $L^2(F_N, \sigma)$. 
Let $u \in W_{\loc}^{1,2}(\Omega) \cap L^{\infty}(\Omega)$ be a weak solution of $Lu=0$ in $\Omega$. 
Let $\vartheta\ge \vartheta_0$ and note that by \eqref{eq:kappa}, we see that $2B_Q \subset B_{Q_0}^*$. Recalling \eqref{eq:normalize} and the fact $X_0 \not\in 4B_{Q_0}^*$, we use \eqref{eq:comparison},  \eqref{eq:doubling}, \eqref{eq:NN}, Harnack's inequality, and the fact that $\pom$ is ADR to conclude that 
\begin{align}\label{eq:w-sig-com}
\frac{\G(X)}{\delta(X)} \simeq \frac{\w(Q)}{\sigma(Q)} \simeq_N 1, 
\end{align}
for all $X \in I^*$ with $I \in \W_Q^{\vartheta,*}$ and $Q \in \D_{\F_N, Q_0}$. This and the definition of $\Omega_{\F_N, Q_0}^\vartheta$ yield  
\begin{align}\label{eq:u-G}
\iint_{\Omega_{\F_N, Q_0}^\vartheta} & |\nabla u(Y)|^2 \delta(Y) dY 
\lesssim_N \iint_{\Omega_{\F_N, Q_0}^\vartheta} |\nabla u(Y)|^2 \G(Y) dY.  
\end{align}
For every $M \geq 1$, we set $\F_{N, M}$ to be the family of maximal cubes of the collection $\F_N$ 
augmented by adding all the cubes $Q \in \D_{Q_0}$ such that $\ell(Q) \leq 2^{-M} \ell(Q_0)$. 
This means that $Q \in \D_{\F_{N,M}, Q_0}$ if and only if $Q \in \D_{\F_N, Q_0}$ and $\ell(Q)>2^{-M}\ell(Q_0)$. 
Observe that $\D_{\F_{N,M}, Q_0} \subset \D_{\F_{N, M'}, Q_0}$ for all $M \leq M'$, and hence 
$\Omega_{\F_{N,M}, Q_0}^\vartheta \subset \Omega_{\F_{N,M'}, Q_0}^\vartheta \subset \Omega_{\F_N, Q_0}^\vartheta$. This, together with the 
monotone convergence theorem, gives 
\begin{align}\label{eq:lim-FM}
\iint_{\Omega_{\F_N, Q_0}^\vartheta} |\nabla u(Y)|^2 \G(Y) dY 
=\lim_{M \to \infty} \iint_{\Omega_{\F_{N,M}, Q_0}^\vartheta} |\nabla u(Y)|^2 \G(Y) dY.  
\end{align}
Invoking \cite[Proposition 3.58]{CHMT}, one has  
\begin{align}\label{eq:int-parts}
\iint_{\Omega_{\F_{N,M}, Q_0}^\vartheta} |\nabla u(Y)|^2 \G(Y) dY 
\lesssim_{N} \sigma(Q_0) \simeq 2^{-k_0 n}, 
\end{align}
where the implicit constants are independent of $M$. 
Consequently, combining \eqref{eq:u-G}, \eqref{eq:lim-FM} and \eqref{eq:int-parts}, we deduce that 
\begin{align}\label{eq:saw-square}
\iint_{\Omega_{\F_N, Q_0}^\vartheta} & |\nabla u(Y)|^2 \delta(Y) dY \leq C_N. 
\end{align}

To continue, we recall the dyadic square function defined in \cite[Section 2.3]{HMU}: 
\begin{align*}
S_{Q_0}^\vartheta u(x) := \left(\iint_{\Gamma_{Q_0}^\vartheta(x)} |\nabla u(Y)|^2 \delta(Y)^{1-n} dY \right)^{1/2}, 
\text{ where } \Gamma_{Q_0}^\vartheta(x) := \bigcup_{x \in Q \in \D_{Q_0}} U_{Q}^\vartheta. 
\end{align*}
Note that if $Q \in \D_{Q_0}$ is so that $Q \cap E_N \neq \emptyset$, then necessarily $Q \in \D_{\F_N, Q_0}$, 
otherwise, $Q \subset Q' \in \F_N$, hence $Q \subset Q_0 \backslash E_N$. In view of \eqref{eq:saw-square}, we have  
\begin{align}\label{eq:SQk-ENk}
\int_{E_N} S_{Q_0}^\vartheta u(x)^2 d\sigma(x) 
&=\int_{E_N} \iint_{\bigcup\limits_{x \in Q \in \D_{Q_0}} U_Q^\vartheta} 
|\nabla u(Y)|^2 \delta(Y)^{1-n} dY d\sigma(x) 
\nonumber \\
&\lesssim \sum_{Q \in \D_{Q_0}} \ell(Q)^{-n} 
\sigma(Q \cap E_N) \iint_{U_Q^\vartheta}  |\nabla u(Y)|^2 \delta(Y) dY 
\nonumber \\
&\lesssim \sum_{Q \in \D_{\F_N, Q_0}} 
\iint_{U_Q^\vartheta}  |\nabla u(Y)|^2 \delta(Y) dY 
\nonumber \\
&\lesssim \iint_{\Omega_{\F_N, Q_0}^\vartheta}  |\nabla u(Y)|^2 \delta(Y) dY 
\leq C_{N}, 
\end{align}
where we have used that the family $\{U_Q^\vartheta\}_{Q\in\D}$ has bounded overlap.  
This along with the last condition in \eqref{eq:EN-FN} yields 
\begin{align}\label{eq:square-FNj}
S_{Q_0}^\vartheta u \in L^2(F_N, \sigma), \qquad\forall\, \vartheta\ge \vartheta_0.
\end{align} 

We next claim that fixed $\alpha>0$, we can find $\vartheta$ sufficiently large depending on $\alpha$ such that for any $r_0 \ll 2^{-k_0}$,  
\begin{align}\label{eq:Sa-SQ}
S_{\alpha}^{r_0}u(x) \leq S_{Q_0}^\vartheta u(x), \quad x \in Q_0. 
\end{align}
It suffices to show $\Gamma_{\alpha}^{r_0}(x) \subset \Gamma_{Q_0}^\vartheta(x)$ for any $x \in Q_0$. 
Indeed, let $Y \in \Gamma_{\alpha}^{r_0}(x)$. Pick $I \in \W$ so that $Y \in I$, and hence, 
$\ell(I) \simeq \delta(Y) \leq |Y-x|<r_0 \ll 2^{-k_0} = \ell(Q_0)$. 
Pick $Q_I \in \D_{Q_0}$ such that $x \in Q_I$ and $\ell(Q_I)=\ell(I) \ll \ell(Q_0)$. 
Thus, one has 
\begin{align*}
\dist(I, Q_I) \leq |Y-x| < (1+\alpha) \delta(Y) \leq C(1+\alpha) \ell(I) = C(1+\alpha) \ell(Q_I).  
\end{align*}
Recalling \eqref{eq:WQ}, if we take $\vartheta\ge \vartheta_0$ large enough so that 
\begin{equation}\label{eq:c0-alpha}
2^\vartheta \geq C(1+\alpha),
\end{equation}  
then $Y \in I \in \W_{Q_I}^\vartheta \subset \W^{\vartheta,*}_{Q_I}$. The latter gives that  
$Y \in U_{Q_I}^\vartheta \subset \Gamma_{Q_0}^\vartheta(x)$ and consequently \eqref{eq:Sa-SQ} holds.  
We should mention that the dependence of $\vartheta$ on $\alpha$ implies that all the sawtooth 
regions $\Omega_{\F_N, Q_0}^\vartheta$ above as well as all the implicit constants depend on $\alpha$.

To complete the proof we note that, it follows from \eqref{eq:square-FNj} and \eqref{eq:Sa-SQ} that 
$S_{\alpha}^{r_0} u \in L^2(F_N, \sigma)$. This together with Remark \ref{remark:truncations} easily yields 
\begin{equation}\label{eq:Sr-u-L2-FN}
S_{\alpha}^r u \in L^2(F_N, \sigma),\quad \text{for any } r>0.  
\end{equation}

We note that the previous argument has been carried out for an arbitrary $Q_0\in \D_{k_0}$. Hence, using  \eqref{eq:E0N-EN}, \eqref{eq:Q-decom}, and \eqref{eq:EN-FN} with $Q_k\in \D_{k_0}$, 
we conclude, with the induced notation, that 
\begin{align}\label{eq:EE-FF}
\partial \Omega = \bigcup_{Q_k \in \D_{k_0}} Q_k 
&=\bigg(\bigcup_{Q_k \in \D_{k_0}} E^k_0\bigg) \bigcup 
\bigg(\bigcup_{Q_k \in \D_{k_0}} \bigcup_{N>C_0} E^k_N \bigg) 
\nonumber \\
&=\bigg(\bigcup_{Q_k \in \D_{k_0}} E^k_0\bigg) \bigcup 
\bigg(\bigcup_{Q_k \in \D_{k_0}} \bigcup_{N>C_0} F^k_N \bigg) 
=: F_0 \cup \bigg(\bigcup_{k, N} F^k_N \bigg), 
\end{align}
where $\sigma(F_0)=0$ and $F^k_N=\partial \Omega \cap \partial \Omega_{\F^k_N, Q_k}$ where each 
$\Omega_{\F^k_N, Q_k} \subset \Omega$ is a bounded 1-sided CAD . Combining \eqref{eq:EE-FF} and \eqref{eq:Sr-u-L2-FN} with $F_N^k$ in place of $F_N$, the proof of 
$\eqref{list:wL} \Rightarrow \eqref{list:Sr-L2}$ is complete.  \qed

\subsection{Proof of \texorpdfstring{$\eqref{list:wL} \Longrightarrow\eqref{list:si-wL-si}$}{(a) implies (b)}}

We borrow some idea from \cite{TZ} and address some small inaccuracies 
that do not affect their conclusion. We follow and use the notation from the proof of $\eqref{list:wL}\Longrightarrow \eqref{list:Sr-L2}$. As before we fix an arbitrary cube $Q_0 \in \D_{k_0}$ and an integer $N>C_0$.  Recall that the family $\F_N$ of stopping cubes is constructed in \eqref{eq:stopping} and $E_N^0=E_N^+\cup E_N^-$ defined in \eqref{eq:E0N-EN}. As we have shown that $\{E_N^+\}_N$, $\{E_N^-\}_N$, and $\{E_N^0\}_N$ are decreasing sequence of sets it is easy to see that 
\begin{align}\label{q34rarefer}
\bigcap_{N>C_0} E_N^0
=
\bigg(\bigcap_{N>C_0} E_N^+\bigg)\bigcup \bigg(\bigcap_{N>C_0} E_N^- \bigg)
\end{align}

To proceed, recall our assumption $\sigma\ll \omega$. Set $h=d\sigma/d\w$ and 
\begin{align}\label{eq:L0-def}
L_0 := \left\{x \in Q_0: \fint_{Q_x} |h(y)-h(x)| \ d\w(y) \to 0,  \D_{Q_0} \ni Q_x \searrow \{x\} \right\}.
\end{align}
By the Lebesgue differentiation theorem for dyadic cubes
\begin{align}\label{eq:L0-points}
\w(Q_0 \backslash L_0)=0
\end{align} 
and, hence,
\begin{equation}\label{eq:sig-Q0-L0}
\sigma(Q_0 \backslash L_0)=0.
\end{equation} 
Then we can write
\[
Q_0
=
\bigg(\bigcup_{N>C_0} (L_0\setminus E_N^0)\bigg)
\bigcup \bigg(\bigcap_{N>C_0} E_N^+\bigg) \bigcup \bigg(\bigcap_{N>C_0} E_N^-\bigg)
\bigcup (Q_0\setminus L_0).
\]
By \eqref{54wt5gt}, \eqref{54wt5gt:alt}, and \eqref{eq:sig-Q0-L0} we then have
\[
\sigma \bigg(\bigcap_{N>C_0} E_N^+\bigg)
=
\sigma \bigg(\bigcap_{N>C_0} E_N^-\bigg)
=
\sigma(Q_0\setminus L_0)=0.
\]
Therefore we just need to show that there exists $C_N>1$ such that 
\begin{align}\label{eq:EEE}
C_N^{-1} \sigma(F) \leq \w_L(F) \leq C_N \sigma(F), \quad \forall\,F \subset L_0\setminus E_N^0.  
\end{align}
Assuming this momentarily, and applying it to every $Q_k\in \D_{k_0}$ we readily get \eqref{list:si-wL-si} using that $\pom=\bigcup_{Q_k\in \D_{k_0}} Q_k$.

Let us then focus on justifying \eqref{eq:EEE}. We claim that for any $x\in L_0\setminus E_N^0$ there holds
\begin{align}\label{rffrw}
\frac1{N}\le \frac{\sigma(Q)}{\w(Q)}\le N,
\qquad
\forall\,Q\in\D_{Q_0}, Q\ni x.
\end{align}
Otherwise, by the maximality of $\F_N^{+}$ or $\F_N^{-}$, one has  $Q\subset Q_j$ for some $Q_j\in\F_N^{+}\cup \F_N^{-}$. Hence $x\in E_N^0$ by \eqref{eq:E0N-EN} which is a contradiction. Using \eqref{eq:L0-def} and \eqref{rffrw} we then conclude that 
\begin{align}\label{stgbgbsrtb}
\frac1{N}\le h(x)\le N,
\qquad
\forall\,x\in L_0\setminus E_N^0.
\end{align}
Thus, for every $F\subset L_0\setminus E_N^0$ we conclude that
\[
\frac1{N}\,\w(F)
\le
\int_{F} h\,d\w
\le
N\,\w(F).
\]
This together with the fact that $h=d\sigma/d\w$ readily implies \eqref{eq:EEE} and \eqref{list:si-wL-si} holds as explained above. \qed

\subsection{Proof of \texorpdfstring{$\eqref{list:Sr}\Longrightarrow\eqref{list:wL}$}{(d) implies (a)}} 
Given $Q_0 \in \D$ and for any $\eta \in (0, 1)$, we define the modified dyadic square function 
\begin{align*}
S_{Q_0}^{\vartheta_0,\eta} u(x) := \left(\iint_{\Gamma_{Q_0}^{\vartheta_0,\eta}(x)}
|\nabla u(Y)|^2 \delta(Y)^{1-n} dY \right)^{1/2}, 
\end{align*}
where the modified non-tangential cone $\Gamma_{Q_0}^{\vartheta_0,\eta}(x)$ is given by 
\begin{align*}
\Gamma_{Q_0}^{\vartheta_0,\eta}(x) := \bigcup_{x \in Q \in \D_{Q_0}} U_{Q, \eta^3}^{\vartheta_0}, \qquad 
U_{Q, \eta^3}^{\vartheta_0} = \bigcup_{\substack{Q' \in \D_Q \\ \ell(Q')>\eta^3 \ell(Q)}} U_{Q'}^{\vartheta_0}. 
\end{align*}
Here we recall that $\vartheta_0$ depends on the 1-sided CAD constants (see Section \ref{section:dyadic}).

The following auxiliary result, whose proof is postponed to  Appendix~\ref{appendix:square}, extends \cite[Lemma~3.10]{CHMT} (see also \cite{KKoPT,KKiPT}). 
\begin{lemma}\label{lem:square}
There exist $0<\eta \ll 1$ (depending only on dimension, the $1$-sided CAD constants and the ellipticity of $L$) 
such that for every $Q_0 \in \D$, and for every Borel set $\emptyset\neq F \subset Q_0$ satisfying 
$\w_L^{X_{Q_0}}(F)=0$, there exists a Borel set $S \subset Q_0$ such that 
the bounded weak solution $u(X)=\w_L^{X}(S)$ satisfies 
\begin{align*}
S_{Q_0}^{\vartheta_0,\eta} u(x)=\infty,\qquad \forall\,x \in F. 
\end{align*}
\end{lemma}

\medskip

Assume that \eqref{list:Sr} holds. In order to prove that $\sigma \ll \w_L$ on $\partial \Omega$, by Lemma \ref{lem:dyadic} 
(a), it suffices to show that 
for any given $Q_0 \in \D$,  
\begin{align}\label{eq:abs-cont-Q0}
F \subset Q_0,\quad \w_L(F)=0 \quad \Longrightarrow \quad \sigma(F)=0. 
\end{align}

Consider then $F \subset Q_0$ with $\w_L(F)=0$. By the mutually absolute continuity between elliptic measures, one has 
$\w_L^{X_{Q_0}}(F )=0$. Lemma \ref{lem:square} applied to $F$ yields that there exists a Borel set $S \subset Q_0$ such that 
$u(X)=\w_L^{X}(S)$, $X\in\Omega$, satisfies 
\begin{align}\label{eq:S-lower}
S_{Q_0}^{\vartheta_0,\eta} u(x)=\infty,  \qquad \forall\,x \in F. 
\end{align}
To continue, we claim that there exist $\alpha_0>0$ and $r>0$ such that 
\begin{align}\label{eq:cone-cone}
\Gamma_{Q_0}^{\vartheta_0,\eta}(x) \subset \Gamma_{\alpha}^{r}(x), 
\qquad  \forall\,x \in Q_0 \text{ and }\forall\,\alpha \geq \alpha_0. 
\end{align}
Indeed, let $Y \in \Gamma_{Q_0}^{\vartheta_0,\eta}(x)$. By definition, there exist $Q \in \D_{Q_0}$ and 
$Q' \in \D_Q$ with $\ell(Q')>\eta^3 \ell(Q)$ such that $Y \in U_{Q'}^{\vartheta_0}$ and $x \in Q$. Then $Y \in I^*$ 
for some $I \in \W_{Q'}^{\vartheta_0,*}$, and hence, 
\begin{align}\label{eq:sidelength}
\delta(Y) \simeq \ell(I) \simeq \ell(Q')\le \ell(Q)<\eta^{-3}\ell(Q').
\end{align}
This further implies that 
\begin{multline}\label{eq:Yx}
|Y-x| \leq \diam(I^*) + \dist(I, x) \leq \diam(I^*) + \dist(I, Q') + \diam(Q) 
\\
\lesssim 2^{k^*}\ell(Q')+\ell(Q)
\lesssim \ell(Q), 
\end{multline}
where $k^*$ depends on the  
$1$-sided CAD constants (see Section \ref{section:dyadic}).
Combining \eqref{eq:sidelength} with \eqref{eq:Yx}, we get 
\begin{align*}
|Y-x| \leq C_1 \ell(Q_0) =: r/2 \quad \text{ and }\quad 
|Y-x| \leq (1+C_{1, \eta}) \delta(Y) =:(1+\alpha_0) \delta(Y), 
\end{align*}
where $C_1$ depends only on the allowable parameters, and $C_{1,\eta}$ depends only on the allowable 
parameters and also on $\eta$. Eventually, this justifies \eqref{eq:cone-cone}.

Combining \eqref{eq:S-lower},  \eqref{eq:cone-cone}, and \eqref{list:Sr}, we readily see that $\sigma(F)=0$ and \eqref{eq:abs-cont-Q0} follows. This completes the proof of  $\eqref{list:Sr}\Longrightarrow\eqref{list:wL}$ and hence that of Theorem \ref{thm:abs-cont}. \qed 

\section{Proof of Theorems \ref{thm:perturbation} and \ref{thm:wL-wLT}}\label{sec:perturbation}  

In order to prove Theorems \ref{thm:perturbation} and \ref{thm:wL-wLT}, we will make use of Theorem 
\ref{thm:abs-cont} and show that the truncated square function is finite $\sigma$-a.e.~for every bounded 
weak solution. Indeed, we are going to show the following more general result, which is a qualitative version of \cite[Theorem 4.13]{CHMT}.   

\begin{theorem}\label{thm:AAAD}
Let $\Omega \subset \R^{n+1}$, $n \ge 2$, be a $1$-sided CAD. There exists $\widetilde{\alpha}_0>0$ 
(depending only on the $1$-sided CAD constants) such that if $L_0 u = -\div(A_0 \nabla u)$ and 
$L_1 u = -\div(A_1 \nabla u)$ are real (not necessarily symmetric) elliptic operators such that $A_0-A_1=A+D$, where 
$A, D \in L^{\infty}(\Omega)$ are real matrices satisfying the following conditions: 
\begin{enumerate}
\item[$(i)$] there exist $\alpha_1 \geq \widetilde{\alpha}_0$ and $r_1>0$ such that
\begin{align}\label{eq:a(X)-delta}
\iint_{\Gamma_{\alpha_1}^{r_1}(x)} a(X)^2 \delta(X)^{-n-1}  dX < \infty, 
\qquad \sigma \text{-a.e. } x \in \partial \Omega, 
\end{align}
where 
\begin{align*}
a(X):=\sup_{Y \in B(X, \delta(X)/2)}|A(Y)|, \qquad X \in \Omega;
\end{align*}
\item[$(ii)$] $D \in \Lip_{\loc}(\Omega)$ is antisymmetric and there exist $\alpha_2 \geq \widetilde{\alpha}_0$ and $r_2>0$ such that 
\begin{align}\label{eq:divCD} 
\iint_{\Gamma_{\alpha_2}^{r_2}(x)} |\div_{C} D(X)|^2 \delta(X)^{1-n} dX < \infty, 
\quad \sigma \text{-a.e. } x \in \partial \Omega; 
\end{align}
\end{enumerate}
then $\sigma \ll \w_{L_0}$ if and only if $\sigma \ll \w_{L_1}$.  
\end{theorem}

\begin{proof}
By symmetry, it suffices to assume that $\sigma \ll \w_{L_0}$ and prove $\sigma \ll \w_{L_1}$. Let 
$u \in W_{\loc}^{1,2}(\Omega) \cap L^{\infty}(\Omega)$ be a weak solution of $L_1u=0$ in $\Omega$ and 
$\|u\|_{L^{\infty}(\Omega)}=1$. Applying Theorem \ref{thm:abs-cont} $\eqref{list:Sr} \Rightarrow \eqref{list:wL}$ 
to $u$,  we are reduced to showing that for some $r>0$, 
\begin{align*}
S_{\alpha_0}^ru(x) < \infty, \qquad \text{for $\sigma $-a.e. } x \in \partial \Omega, 
\end{align*}
where $\alpha_0$ is given in Theorem 
\ref{thm:abs-cont}.
Proceeding as in Section \ref{sec:a-c} and invoking \eqref{eq:Sa-SQ}, it suffices to see that for every fixed 
$Q_0 \in \D_{k_0}$ and for some fixed large $\vartheta$ (which depends on $\alpha_0$ and hence solely on the 1-sided CAD constants) one has 
\begin{align}\label{eq:Q0-SQ0u-EN}
Q_0 = \bigcup_{N \geq 0} \widehat{E}_N,\quad \sigma(\widehat{E}_0)=0
\quad\text{and}\quad S_{Q_0}^\vartheta u \in L^2(\widehat{E}_N,\sigma), \ \forall\,N \geq 1.  
\end{align}

Fix then $Q_0 \in \D_{k_0}$, where $k_0$ is given in the beginning of Section \ref{sec:a-c}. We use the 
normalization in \eqref{eq:normalize} with $L=L_0$ and the family $\F_N$ of stopping cubes constructed in \eqref{eq:stopping}. Set
\begin{align*}
S_{Q_0}\gamma^\vartheta(x) := \bigg(\sum_{x \in Q \in \D_{Q_0}} \gamma_{Q} ^\vartheta\bigg)^{1/2},
\end{align*}
where for every $Q\in\D_{Q_0}$ we write
\begin{align*}
\gamma_Q^\vartheta:=
\iint_{U^{\vartheta, *}_Q} a(X)^2 \delta(X)^{-n-1}  dX 
+ \iint_{U^{\vartheta, *}_Q} |\div_{C} D(X)|^2 \delta(X)^{1-n} dX. 
\end{align*}
We claim that there exist $\widetilde{\alpha}_0>0$ and $\widetilde{r}>0$ such that  
\begin{equation}\label{eq:TQ-Tar}
\Gamma^{\vartheta,*}_{Q_0}(x)
:=\bigcup_{x \in Q \in \D_{Q_0}} U^{\vartheta, *}_{Q}
 \subset \Gamma_{\widetilde{\alpha}_0}^{\widetilde{r}}(x),\quad x \in \partial \Omega. 
\end{equation}
Indeed, let $Y \in \Gamma^{\vartheta,*}_{Q_0}(x)$. Then, there exists $Q \in \D_{Q_0}$ with $Q \ni x$ and $I \in \W^{\vartheta,*}_Q$ 
such that $Y \in I^{**}$. Using these, one has 
\begin{align*}
|Y-x| \leq \diam(I^{**}) + \dist(I, Q) + \diam(Q) \lesssim_\vartheta \ell(Q) \simeq_\vartheta \ell(I) \simeq \delta(Y), 
\end{align*}
which implies  
\begin{equation*}
|Y-x| < C_1 \delta(Y)=:(1+\widetilde{\alpha}_0) \delta(Y) \quad\text{and}\quad 
|Y-x| < C_2 \ell(Q_0)=C_2 2^{-k_0} =: \widetilde{r}, 
\end{equation*}
where both $C_1$ and $C_2$ depend only on the allowable parameters ---note that they depend on $\vartheta$, hence on the 1-sided CAD constants. Thus, \eqref{eq:TQ-Tar} holds for the choice 
of $\widetilde{\alpha}_0$ and $\widetilde{r}$, and as a result 
\begin{align}\label{eq:SQ0-a.e.}
&S_{Q_0}\gamma^\vartheta(x)^2 \quad\lesssim
\iint_{\Gamma^{\vartheta,*}_{Q_0}(x) } a(X)^2 \delta(X)^{-n-1}  dX 
+ \iint_{\Gamma^{\vartheta,*}_{Q_0}(x) } |\div_{C} D(X)|^2 \delta(X)^{1-n} dX 
\nonumber\\
&\quad\le  \iint_{\Gamma_{{\alpha}_1}^{\max\{\widetilde{r}, r_1\}}(x)} a(X)^2 \delta(X)^{-n-1}  dX 
+ \iint_{\Gamma_{{\alpha}_2}^{\max\{\widetilde{r},r_2\}}(x)} |\div_{C} D(X)|^2 \delta(X)^{1-n} dX 
\nonumber\\
&\quad< \infty, \quad \text{for $\sigma $-a.e.~} x \in  Q_0, 
\end{align}
where we have used that the fact that the family $\{U^{\vartheta,*}_Q\}_{Q \in \D}$ has bounded overlap, that  $\alpha_1, \alpha_2\ge \widetilde{\alpha}_0$ and the last estimate follows from \eqref{eq:a(X)-delta}, \eqref{eq:divCD} together with Remark \ref{remark:truncations}.

Given $N>C_0$ ($C_0$ is the constant that appeared in Section \ref{sec:a-c}), 
let $\widetilde{\F}_N\subset \D_{Q_0}$ be the collection of maximal cubes (with respect to the inclusion) $Q_j \in \D_{Q_0}$ such that 
\begin{align}\label{eq:FN-stopping}
\sum_{Q_j \subset Q \in \D_{Q_0}} \gamma_{Q}^\vartheta > N^2.
\end{align}
Observe that  
\begin{equation}\label{eq:Sa<N}
S_{Q_0}\gamma^\vartheta(x) \leq N, \qquad \forall\,x \in Q_0 \backslash \bigcup_{Q_j \in \widetilde{\F}_N}Q_j.  
\end{equation}
Otherwise, there exists a cube $Q_x\ni x$ such that 
$\sum_{Q_x \subset Q \in  \D_{Q_0}} \gamma_Q^\vartheta > N^2$, hence $x \in Q_x \subset Q_j$ for some 
$Q_j \in \widetilde{\F}_N$, which is a contradiction. 

We next set 
\begin{align}\label{eq:E0-tilde-measure:a}
\widetilde{E}_0 := 
\bigcap_{N>C_0} \widetilde{E}_N^0  
:=
\bigcap_{N>C_0} \bigg(\bigcup_{Q_j \in \widetilde{\F}_N} Q_j\bigg).  
\end{align}
Let $x \in \widetilde{E}^0_{N+1}$. Then there exists $Q_x \in \widetilde{\F}_{N+1}$ such that $x \in Q_x$. By \eqref{eq:FN-stopping}, one has
\begin{align*}
\sum_{Q_x \subset Q \in \D_{Q_0}} \gamma_{Q}^\vartheta > (N+1)^2>N^2.
\end{align*}
Therefore, the maximality of the cubes in $\widetilde{\F}_N$ gives that $Q_x \subset Q'_x$ for some $Q'_x \in \widetilde{\F}_N$ with $x \in Q'_x \subset \widetilde{E}_0^N$. This shows that $\{\widetilde{E}_N^0\}_N$ is a decreasing sequence of sets, and since $\widetilde{E}_N^0\subset Q_0$ for every $N$ we conclude that 
\begin{align*}
\omega(\widetilde{E}_0 )=\lim_{N\to\infty} \omega(\widetilde{E}_N^0),
\qquad
\sigma(\widetilde{E}_0 )=\lim_{N\to\infty} \sigma(\widetilde{E}_N^0).
\end{align*}
Note that for every $N>C_0$, if $x \in \widetilde{E}_0$ there exists $Q_x^{N} \in \widetilde{\F}_N$ such that 
$Q_x^N \ni x$. By the definition of $\widetilde{\F}_N$, we have 
\begin{align*}
S_{Q_0}\gamma^\vartheta(x)^2 = \sum_{x \in Q \in \D_{Q_0}} \gamma_Q^\vartheta 
\geq  \sum_{Q_x^N \subset Q \in \D_{Q_0}} \gamma_Q^\vartheta > N^2, 
\end{align*}
and, therefore,
 
\begin{multline}\label{eaf4a3f4}
\sigma(\widetilde{E}_0 )
=
\lim_{N\to\infty} \sigma(\widetilde{E}_N^0)
\le \lim_{N\to\infty}  \sigma(\{x \in Q_0: S_{Q_0}\gamma^\vartheta(x)>N\})
\\=\sigma(\{x \in Q_0: S_{Q_0}\gamma^\vartheta(x)=\infty\})=0,
\end{multline}
by \eqref{eq:SQ0-a.e.}.  

To proceed, let $\widehat{\F}_N$ be the collection of maximal, hence pairwise disjoint, cubes in $\F_N \cup \widetilde{\F}_N$. Note that $\D_{\widehat{\F}_N, Q_0} \subset \D_{\F_N, Q_0} \cap \D_{\widetilde{\F}_N, Q_0}$.  
This along with \eqref{eq:NN} yields 
\begin{align}\label{eq:NN-Fhat}
\frac{1}{N}\leq \frac{\w(Q)}{\sigma(Q)} \leq N, \quad \forall\,Q \in \D_{\widehat{\F}_N, Q_0}. 
\end{align}
We next set 
\begin{align}\label{eq:E0-tilde-measure:aaa}
\widehat{E}_0 := 
\bigcap_{N>C_0} \widehat{E}_N^0  
:=
\bigcap_{N>C_0	} \bigg(\bigcup_{Q_j \in \widehat{\F}_N} Q_j\bigg).  
\end{align}
Note that $\widehat{\F}_N\subset\F_N \cup \widetilde{\F}_N$ and also that if $Q\in \F_N \cup \widetilde{\F}_N$ then there exists $Q'\in \F_N \cup \widetilde{\F}_N$ so that $Q\subset Q'$. This shows that
$\widehat{E}_N^0  = {E}_N^0  \cup \widetilde{E}_N^0$, where ${E}_N^0$ and $\widetilde{E}_N^0$ are defined in \eqref{eq:E0N-EN} and \eqref{eq:E0-tilde-measure:a} respectively. As we showed that $\{E_N^0\}_N$ and $\{\widetilde{E}_N^0\}_N$ are decreasing sequence of sets, then so is $\{\widehat{E}_N^0\}_N$. This together with the fact that $\widehat{E}_N^0\subset Q_0$ lead to 
\begin{align*}
\sigma(\widehat{E}_0)
=\lim_{N \to\infty}  \sigma(\widehat{E}_N^0)
\le\lim_{N \to\infty}  \sigma({E}_N^0)
+\lim_{N \to\infty}  \sigma(\widetilde{E}_N^0)
=0, 
\end{align*}
as shown in \eqref{34fravcrv} and \eqref{eaf4a3f4}, hence $\sigma(\widehat{E}_0)=0$.

Next we write  
\begin{align}\label{eq:Q0-Ehat}
Q_0=
\widehat{E}_0\cup \bigg(\bigcup_{N>C_0} \widehat{E}_N \bigg)
:= 
\widehat{E}_0\cup \bigg(\bigcup_{N>C_0} (Q_0\setminus \widehat{E}_N^0) \bigg).
\end{align}
Therefore, to get \eqref{eq:Q0-SQ0u-EN}, we are left with proving
\begin{equation}\label{eq:SQ0u-EN-tilde}
S_{Q_0}^\vartheta u \in L^2(\widehat{E}_N, \sigma), \quad \forall\,N>C_0. 
\end{equation} 
With this goal in mind,  we apply \eqref{eq:NN-Fhat}, 
\eqref{eq:Q0-Ehat} and proceed as in the proof of \eqref{eq:SQk-ENk} and \eqref{eq:u-G}, to conclude that
\begin{align}\label{eq:S-delta-G}
\int_{\widehat{E}_N} S_{Q_0}^\vartheta u(x)^2 d\sigma(x) 
\lesssim \iint_{\Omega_{\widehat{\F}_N, Q_0}^\vartheta}  |\nabla u|^2 \delta \ dY 
\lesssim_N \iint_{\Omega_{\widehat{\F}_{N}, Q_0}^\vartheta}  |\nabla u|^2 \G \ dY.  
\end{align}
As in Section \ref{sec:a-c}, for every $M \geq 1$, we consider the pairwise disjoint collection $\widehat{\F}_{N,M}$ that is the family of maximal cubes of the collection $\widehat{\F}_N$ augmented by adding all the cubes 
$Q \in \D_{Q_0}$ such that $\ell(Q) \leq 2^{-M} \ell(Q_0)$. In particular, $Q \in \D_{\widehat{\F}_{N,M}, Q_0}$ 
if and only if $Q \in \D_{\widehat{\F}_N, Q_0}$ and $\ell(Q)>2^{-M}\ell(Q_0)$.  Moreover, 
$\D_{\widehat{\F}_{N,M}, Q_0} \subset \D_{\widehat{\F}_{N,M'}, Q_0}$ for all $M \leq M'$, and hence 
$\Omega_{\widehat{\F}_{N,M}, Q_0}^\vartheta \subset \Omega_{\widehat{\F}_{N,M'}, Q_0}^\vartheta \subset \Omega_{\widehat{\F}_N, Q_0}^\vartheta$.  Then the monotone convergence theorem implies 
\begin{align}\label{eq:JM-lim}
\iint_{\Omega_{\widehat{\F}_{N}, Q_0}^\vartheta}  |\nabla u|^2 \G \ dY 
=\lim_{M \to \infty}\iint_{\Omega_{\widehat{\F}_{N,M}, Q_0}^\vartheta}  |\nabla u|^2 \G \ dY
=:
\lim_{M \to \infty} \mathcal{J}_M.
\end{align} 
To continue with the proof, we are going to follow \cite[Proof of Proposition 4.18]{CHMT}. Let $\Psi\in C^\infty_c(\R^{n+1})$ be the smooth cut-off function associated with the sawtooth domain $\Omega_{\widehat{\F}_{N, M}, Q_0}^\vartheta$ (see \cite[Lemma 3.61]{CHMT} or \cite[Lemma 4.44]{HMT1}) and note that since $\Psi\gtrsim 1$ in $\Omega_{\widehat{\F}_{N,M}}^{\vartheta}$ we have
\begin{align}\label{eq:JM-lim-tilde}
\mathcal{J}_M \lesssim \widetilde{\mathcal{J}}_M 
:= \iint_{\Omega}  |\nabla u|^2 \G \Psi^2\ dY. 
\end{align}
Note that $\widetilde{\mathcal{J}}_M<\infty$ because $\supp \Psi\subset \overline{\Omega_{\widehat{\F}_{N, M}, Q_0}^{\vartheta,*}} \subset \Omega$ 
and $u \in W^{1,2}_{\loc}(\Omega)$. A careful examination of \cite[Proof of Proposition 4.18]{CHMT} gives
\begin{multline*}
\widetilde{\mathcal{J}}_M \lesssim_N \sigma(Q_0) +  \widetilde{\mathcal{J}}_M^{\frac12} 
\left(\iint_{\Omega^{\vartheta,*}_{\widehat{\F}_{N,M}, Q_0}} \frac{a(X)^2}{\delta(X)} dX \right)^{\frac12}
\\
+ \widetilde{\mathcal{J}}_M^{\frac12} \sigma(Q_0)^{\frac12} 
+ \sigma(Q_0)^{\frac12} 
\left(\iint_{\Omega^{\vartheta,*}_{\widehat{\F}_{N,M}, Q_0}} |\div_C D(X)|^2 \delta(X) dX \right)^{\frac12}. 
\end{multline*}
In turn, applying Young's inequality and hiding, we readily get 
\begin{align}\label{eq:JM}
\widetilde{\mathcal{J}}_M \lesssim_N \sigma(Q_0) + 
\iint_{\Omega^{\vartheta,*}_{\widehat{\F}_N, Q_0}} \frac{a(X)^2}{\delta(X)} dX 
+ \iint_{\Omega^{\vartheta,*}_{\widehat{\F}_N, Q_0}} |\div_C D(X)|^2 \delta(X) dX, 
\end{align}
where the implicit constant is independent of $M$. 
Collecting \eqref{eq:S-delta-G}, \eqref{eq:JM-lim}, \eqref{eq:JM-lim-tilde}, and \eqref{eq:JM}, we obtain  
\begin{align}\label{eq:S-dis}
&\int_{\widehat{E}_N} S_{Q_0}^\vartheta u(x)^2 d\sigma(x) 
\lesssim \sigma(Q_0) + \iint_{\Omega^{\vartheta,*}_{\widehat{\F}_N, Q_0}} \frac{a(X)^2}{\delta(X)} dX 
+ \iint_{\Omega^{\vartheta,*}_{\widehat{\F}_N, Q_0}} |\div_C D(X)|^2 \delta(X) dX
\nonumber \\
&\qquad\leq\sigma(Q_0) + \sum_{Q \in \D_{\widehat{\F}_N, Q_0}} 
\bigg(\iint_{U^{\vartheta,*}_Q} \frac{a(X)^2}{\delta(X)} dX
+
\iint_{U^{\vartheta,*}_Q} |\div_C D(X)|^2 \delta(X)dX\bigg)
\nonumber \\
&\qquad\lesssim \sigma(Q_0) + \sum_{Q \in \D_{\widehat{\F}_N, Q_0}} \gamma_Q^\vartheta \sigma(Q),  
\end{align}
where we used that $\sigma(Q) \simeq \ell(Q)^n \simeq \delta(X)^n$ for every $X \in U^{\vartheta,*}_Q$. 
On the other hand, 
\begin{multline}\label{eq:dis}
\sum_{Q \in \D_{\widehat{\F}_N, Q_0}} \gamma_Q^\vartheta \sigma(Q) 
=\int_{Q_0} \sum_{x \in Q \in \D_{\widehat{\F}_N, Q_0}} \gamma_Q^\vartheta d\sigma(x)
\\
\leq \int_{\widehat{E}_N} S_{Q_0} \gamma^\vartheta(x)^2 d\sigma(x) 
+ \sum_{Q_j \in \widehat{\F}_N} \sum_{Q \in \D_{\widehat{\F}_N, Q_0}} \gamma_Q^\vartheta\,\sigma(Q\cap Q_j).  
\end{multline}
As observed above $\widetilde{E}_N^0\subset \widehat{E}_N^0$, hence, \eqref{eq:Sa<N} leads to 
\begin{align}\label{eq:dis-1}
\int_{\widehat{E}_N} S_{Q_0} \gamma^\vartheta(x)^2 d\sigma(x) 
\leq N^2 \sigma(Q_0).  
\end{align}
In order to control the second term in \eqref{eq:dis},  we fix $Q_j \in \widehat{\F}_N$. Note that if 
$Q \in \D_{\widehat{\F}_N, Q_0}$ is so that
$Q \cap Q_j \neq \emptyset$ then necessarily $Q_j \subsetneq Q$.
Write $\widehat{Q}_j$ for the dyadic father of $Q_j$, that is,  $\widehat{Q}_j$ is the unique dyadic cube containing $Q_j$ 
with $\ell(\widehat{Q}_j)=2\ell(Q_j)$. 
We claim that 
\begin{equation}\label{eq:QjQ-N2}
\sum_{\widehat{Q}_j \subset Q \in \D_{Q_0}} \gamma_Q^\vartheta  =\sum_{Q_j \subsetneq Q \in \D_{Q_0}} \gamma_Q^\vartheta \leq N^2. 
\end{equation}
Otherwise, recalling the construction of $\widetilde{\F}_N$ in \eqref{eq:FN-stopping}, it follows that $\widehat{Q}_j \subset Q'$ for some $Q' \in \widetilde{\F}_N$. From the definition of $\widehat{\F}_N$, we then have that $Q' \subset Q''$ for some $Q'' \in  \widehat{\F}_N$. Consequently, $Q_j \subsetneq Q''$ with $Q_j, Q'' \in \widehat{\F}_N$ contradicting the maximality of the family $\widehat{\F}_N$. Then it follows from \eqref{eq:QjQ-N2} that 
\begin{multline}\label{43r43r4}
\sum_{Q_j \in \widehat{\F}_N} \sum_{Q \in \D_{\widehat{\F}_N, Q_0}} \gamma_Q^\vartheta \sigma(Q \cap Q_j)  
=\sum_{Q_j \in \widehat{\F}_N} \sigma(Q_j)  \sum_{Q_j \subsetneq Q \in \D_{Q_0}} \gamma_Q^\vartheta 
\\
\leq N^2 \sum_{Q_j \in \widehat{\F}_N} \sigma(Q_j) 
\leq N^2 \sigma \bigg(\bigcup_{Q_j \in \widehat{\F}_N} Q_j \bigg) 
\leq N^2 \sigma(Q_0).   
\end{multline}
Collecting \eqref{eq:S-dis}, \eqref{eq:dis}, \eqref{eq:dis-1}, and \eqref{43r43r4}, we deduce that  
\begin{align*}
\int_{\widehat{E}_N} S_{Q_0}u(x)^2 d\sigma(x) 
\leq C_N \sigma(Q_0)  \simeq C_N 2^{-k_0 n}.
\end{align*}
This shows \eqref{eq:SQ0u-EN-tilde} and completes the proof of Theorem \ref{thm:AAAD}. 
\end{proof}

Now let us see how we deduce Theorems \ref{thm:perturbation} and \ref{thm:wL-wLT} from Theorem \ref{thm:AAAD}. 

\begin{proof}[Proof of Theorem \ref{thm:perturbation}]
Let $L_0$ and $L_1$ be the elliptic operators given in Theorem \ref{thm:perturbation}. If we take $A=A_0-A_1$ and $D=0$ in Theorem \ref{thm:AAAD}, then \eqref{eq:a(X)-delta} coincides with the assumption \eqref{eq:rhoAA} and \eqref{eq:divCD} holds automatically. Therefore, Theorem \ref{thm:perturbation} immediately follows from Theorem \ref{thm:AAAD}.  
\end{proof}

\begin{proof}[Proof of Theorem \ref{thm:wL-wLT}]
Let $A$ be the matrix as stated in Theorem \ref{thm:wL-wLT}. If we take $A_0=A$, $A_1=A^{\top}$, $\widetilde{A}=0$ and 
$D=A-A^{\top}$ in Theorem \ref{thm:AAAD}, then one has $A_0-A_1=\widetilde{A}+D$ with $D \in \Lip_{\loc}(\Omega)$ antisymmetric, \eqref{eq:a(X)-delta} holds trivially and  \eqref{eq:divCD} agrees with \eqref{eq:divCAA}. Thus, Theorem \ref{thm:AAAD} implies that $\sigma \ll \w_L$ if and only if $\sigma \ll \w_{L^{\top}}$. 

Similarly, the conclusion that $\sigma \ll \w_L$ if and only if $\sigma \ll \w_{L^{\rm sym}}$ follows if we set $A_0=A$, $A_1=(A+A^{\top})/2$, $\widetilde{A}=0$ and $D=(A-A^{\top})/2$. 
\end{proof}

\appendix

\section{Extending the construction of Kenig, Kirchheim, Pipher, Toro: Proof of Lemma~\ref{lem:square}}\label{appendix:square}

In this appendix we prove Lemma~\ref{lem:square}. We will follow the construction in \cite[Section 3]{CHMT}  which in turn extends that of \cite{KKiPT} (see also \cite{KKoPT}). In those scenarios the set $F$ is sufficiently small, that is, it satisfies $\w_L^{X_{Q_0}}(F)\le \beta\w_L^{X_{Q_0}}Q_0)$ and it is shown that there is a set $S_\beta$ so that $u_\beta(X)=\omega_{L}^X(S_\beta)$, $X\in\Omega$, satisfies $S_{Q_0}^{\vartheta_0,\eta} u(x)^2 \gtrsim_\eta \log(\beta^{-1})$ for every $x\in F$. Here we obtain the limiting case $\beta=0$.

We start with some definition and some auxiliary result: 

\begin{definition}\label{def:good-cover}
	Let $E \subset \R^{n+1}$ be an $n$-dimensional ADR set. Fix $Q_0 \in \D(E)$ (cf.~Lemma~\eqref{lem:dyadic}) and let $\mu$ be a regular Borel measure on $Q_0$.  Given $\varepsilon_0 \in (0,1)$ and a Borel set $\emptyset\neq F \subset Q_0$, a good $\varepsilon_0$-cover of $F$ with respect to $\mu$, of length $k \in \N$, is a collection $\{\mathcal{O}_{\ell}\}_{\ell=1}^k$ of Borel subsets of $Q_0$, together with pairwise disjoint families $\mathcal{F}_{\ell} =\{Q^{\ell}\} \subset \D_{Q_0}$, $1\le\ell \le  k$, such that the following hold:
	\begin{list}{\textup{(\theenumi)}}{\usecounter{enumi}\leftmargin=1cm \labelwidth=1cm \itemsep=0.2cm 
			\topsep=.2cm \renewcommand{\theenumi}{\alph{enumi}}}
		
		\item $F \subset \mathcal{O}_k \subset \mathcal{O}_{k-1} \subset \dots \subset \mathcal{O}_2 \subset \mathcal{O}_1 \subset Q_0$. 
		
		\item $\mathcal{O}_{\ell}=\bigcup_{Q^{\ell} \in \F_{\ell}} Q^{\ell}$, for every $1 \le \ell \le k$.
		
		\item $\mu(\mathcal{O}_{\ell} \cap Q^{\ell-1}) \le \varepsilon_0 \mu(Q^{\ell-1})$,  for each $Q^{\ell-1} \in \F_{\ell-1}$ and $2\le \ell \le k$. 
		
	\end{list}
Analogously,  a good $\varepsilon_0$-cover of $F$ with respect to $\mu$, of length $\infty$, is a collection $\{\mathcal{O}_{\ell}\}_{\ell=1}^\infty$ of Borel subsets of $Q_0$, together with pairwise disjoint families $\mathcal{F}_{\ell} =\{Q^{\ell}\} \subset \D_{Q_0}$, $\ell\ge 1$, such that the following hold:
\begin{list}{\textup{(\theenumi)}}{\usecounter{enumi}\leftmargin=1cm \labelwidth=1cm \itemsep=0.2cm 
		\topsep=.2cm \renewcommand{\theenumi}{\alph{enumi}}}
	
	\item $F \subset\dots\subset \mathcal{O}_k \subset \mathcal{O}_{k-1} \subset \dots \subset \mathcal{O}_2 \subset \mathcal{O}_1 \subset Q_0$. 
	
	\item $\mathcal{O}_{\ell}=\bigcup_{Q^{\ell} \in \F_{\ell}} Q^{\ell}$, for every $\ell\ge 1$,
	
	\item $\mu(\mathcal{O}_{\ell} \cap Q^{\ell-1}) \le \varepsilon_0 \mu(Q^{\ell-1})$,  for each $Q^{\ell-1} \in \F_{\ell-1}$ and $\ell \ge 2$. 
	
\end{list}
\end{definition}

\medskip

\begin{remark}\label{remark:good-cover-iter}
In the previous definition we implicitly assume that $F\cap Q^\ell\neq\emptyset$ for every  $Q^{\ell}\in\F_\ell$ and for all $1\le\ell\le k$ if the length is $k$, or all $\ell\ge 1$ if the length is infinity. Otherwise, we can remove all the cubes $Q^\ell$ for which $F\cap Q^\ell=\emptyset$, and all the required conditions clearly hold. 

Observe also that if $\{\mathcal{O}_{\ell}\}_{\ell=1}^k$ is a good $\varepsilon_0$-cover of $F$ then, by Definition \ref{def:good-cover},  we have  for every $2\le \ell\le k$
	\begin{align*}
		\mu(\mathcal{O}^1_{\ell}) = \mu(\mathcal{O}^1_{\ell} \cap \mathcal{O}^1_{\ell-1}) = \sum_{Q \in \F^1_{\ell-1}} \mu(\mathcal{O}^1_{\ell} \cap Q)  
		\le \varepsilon_0 \sum_{Q \in \F^1_{\ell-1}} \mu(Q) = \varepsilon_0 \mu(\mathcal{O}^1_{\ell-1}). 
	\end{align*}
	Iterating this for every  $2\le \ell\le k$ we conclude that
	\begin{align}\label{eq:OO}
		\mu(\mathcal{O}^1_{\ell}) 
		\le 
		\varepsilon_0 \mu(\mathcal{O}^1_{\ell-1})
		\le
		\varepsilon_0^2 \mu(\mathcal{O}^1_{\ell-2}) 
		\le 
		\dots 
		\le \varepsilon_0^{\ell-1} \mu(\mathcal{O}^1_1) \le \varepsilon_0^{\ell-1} \mu(Q_0).
	\end{align}
	\end{remark}
	
\medskip

\begin{lemma}[{\cite[Lemma~3.5]{CHMT}}]\label{lem:cover}
	Let $E \subset \R^{n+1}$ be an $n$-dimensional ADR set. Let $\mu$ be a regular Borel measure on $Q_0$ and assume that it is dyadically doubling on $Q_0$, that is, there exists $C_{\mu} \ge 1$ such that $\mu(Q^*) \le C_{\mu} \mu(Q)$ for every $Q \in \D_{Q_0} \setminus \{Q_0\}$, with $Q^* \supset Q$ and $\ell(Q^*) = 2\ell(Q)$ (i.e., $Q^*$ is the ``dyadic parent" of $Q$). For every $0<\varepsilon_0 \le e^{-1}$, if $\emptyset\neq F \subset Q_0$ with $\mu(F) \le \alpha \mu(Q_0)$ and $0<\alpha \le \varepsilon_0^2/(2C_{\mu}^2)$ then $F$ has a good $\varepsilon_0$-cover with respect to $\mu$ of length $k_0=k_0(\alpha, \varepsilon_0, C_\mu) \in \N$, $k_0 \ge 2$, which satisfies $k_0 \approx \frac{\log \alpha^{-1}}{\log \varepsilon_0^{-1}}$. In particular, if $\mu(F)=0$, then $F$ has a good $\varepsilon_0$-cover of arbitrary length. 
\end{lemma}

We would like to mention that in the case $\mu(F)=0$ this result gives an $\varepsilon_0$-cover of arbitrary length. Our goal is to show that in such an scenario one can iterate the construction and construct an $\varepsilon_0$-cover of infinite length:
\begin{lemma}\label{lem:infinite}
	Let $E \subset \R^{n+1}$ be an $n$-dimensional ADR set and fix $Q_0 \in \D(E)$. Let $\mu$ be a regular Borel measure on $Q_0$ and assume that it is dyadically doubling on $Q_0$. For every $0<\varepsilon_0 \le e^{-1}$, if $\emptyset\neq F \subset Q_0$ with $\mu(F)=0$, then $F$ has a good $\varepsilon_0$-cover of length $\infty$. 
\end{lemma}

\begin{proof}
	We are going to iterate Lemma~\ref{lem:cover}. Given $0<\varepsilon_0 \le e^{-1}$ and $0<\alpha< \varepsilon_0^2/(2C_{\mu}^2)$, let $k_0=k_0(\alpha,\varepsilon_0,C_\mu)\in\N$, $k_0\ge 2$, be the value from Lemma~\ref{lem:cover} so that $k_0 \approx \frac{\log \alpha^{-1}}{\log \varepsilon_0^{-1}}$. 
	 Let $F \subset Q_0$ with $\mu(F)=0$. Using that $\mu(F)<\alpha \mu(Q_0)$, Lemma~\ref{lem:cover} gives $\{\mathcal{O}^1_{\ell}\}_{\ell=1}^{k_0}$, a good $\varepsilon_0$-cover of length $k_0$  of $F$ with $\F_\ell^1\subset\D_{Q_0}$ the associated families of pairwise disjoint cubes. This is the first generation in the construction.
	
	To obtain the second generation take an arbitrary $Q \in \F^1_{k_0}$ and note that $\mu(F \cap Q)=0<\alpha \mu(Q)$. We apply again Lemma \ref{lem:cover} in $Q$ to $F \cap Q$ (which is not empty by Remark~\ref{remark:good-cover-iter}) and obtain a good $\varepsilon_0$-cover $\{\widetilde{\mathcal{O}}^2_{\ell}(Q)\}_{\ell=1}^{k_0}$ of $F \cap Q$ with associated families of pairwise disjoint cubes	$\widetilde{\F}_\ell^2(Q)\subset\D_Q$, $1\le\ell\le k_0$, and so that
	\[
	F\cap Q \subset \widetilde{\mathcal{O}}^2_{k_0}(Q) \subset \widetilde{\mathcal{O}}^2_{k_0-1}(Q)  \subset \dots \subset \widetilde{\mathcal{O}}^2_{2}(Q) \subset \widetilde{\mathcal{O}}^2_1(Q) \subset Q\in \F^1_{k_0}
	\]
	Write $\widetilde{\mathcal{O}}^2_\ell:=\bigcup_{Q \in \F^1_{k_0}} \widetilde{\mathcal{O}}^2_\ell (Q)$ and    $\widetilde{\F}_\ell^2=\bigcup_{Q \in \F^1_{k_0}} \widetilde{\mathcal{F}}^2_\ell (Q)$   for $1\le\ell\le k_0$. Since for each $Q \in \F^1_{k_0}$ and for each $1\le \ell\le k_0$ the family $\widetilde{\F}_\ell^2(Q)\subset\D_Q$ is pairwise disjoint, and  the family $\F^1_{k_0}$ is also pairwise disjoint  we easily conclude that $\widetilde{\F}_\ell^2$ is a pairwise disjoint family. Besides, 
	\begin{multline}\label{eq:FOO}
		F = F\cap \mathcal{O}_{k_0}^1=\bigcup_{Q \in \F^1_{k_0}} F\cap Q\subset 		
		\widetilde{\mathcal{O}}^2_{k_0} \subset \widetilde{\mathcal{O}}^2_{k_0-1}  \subset \dots \subset \widetilde{\mathcal{O}}^2_{2} \subset \widetilde{\mathcal{O}}^2_1
		\\		
		\subset \bigcup_{Q \in \F^1_{k_0}} Q= \mathcal{O}_{k_0}^1
		\subset \mathcal{O}^1_{k_0-1} \subset \dots \subset \mathcal{O}^1_1 \subset Q_0. 
	\end{multline}
	Set $\mathcal{O}^2_\ell:=\mathcal{O}^1_\ell$ for $1\le \ell\le  k_0$ and $\mathcal{O}^2_\ell:=\widetilde{\mathcal{O}}^2_{\ell-k_0+1}$ for $k_0+1\le \ell\le  2(k_0-1)+1$. Write $\F^2_\ell\subset\D_{Q_0}$ for the associated families of pairwise disjoint dyadic cubes  for $1\le\ell\le 2(k_0-1)+1$. Our goal is to show that  $\{\mathcal{O}^2_{\ell}\}_{\ell=1}^{2(k_0-1)+1}$ is a good $\varepsilon_0$-cover of $F$ whose length is $2(k_0-1)+1$. By \eqref{eq:FOO} and the previous construction, (a) and (b) in Definition~\ref{def:good-cover} clearly hold. We then need to verify (c). With this goal in mind we note that since $\{\mathcal{O}^1_\ell\}_{\ell=1}^{k_0}$ is a good $\varepsilon_0$-cover, we obtain 
	\begin{align*}
	\mu(\mathcal{O}^2_{\ell} \cap Q)=\mu(\mathcal{O}^1_{\ell} \cap Q) \le \varepsilon_0 \mu(Q), \quad\forall\, Q \in \F^2_{\ell-1}=\F^1_{\ell-1}, \, 2\le \ell \le k_0. 
	\end{align*}
	Also, if $Q \in \F^2_{k_0}=\F^1_{k_0}$ then \eqref{eq:OO} in  Remark~\ref{remark:good-cover-iter} applied to $\widetilde{\mathcal{O}}^2_2(Q)$ gives
	\begin{align*}
	\mu(\mathcal{O}^2_{k_0+1} \cap Q)
	=
	\mu(\widetilde{\mathcal{O}}^2_{2} \cap Q)
	=
	\mu(\widetilde{\mathcal{O}}^2_2(Q)) 
	\le \varepsilon_0 \mu(Q). 
	\end{align*} 
	On the other hand, let  $k_0+2\le \ell\le 2(k_0-1)+1$ and let $Q \in \F^2_{\ell-1}=\widetilde{\F}^2_{\ell-k_0}$. By construction, there exists $Q'\in \F_{k_0}^1$ so that
	 $Q \in\widetilde{\F}^2_{\ell-k_0}(Q')\subset \D_{Q'}$. Then, 
	\begin{align*}
	\mu(\mathcal{O}^2_{\ell} \cap Q)
	=
	\mu(\widetilde{\mathcal{O}}^2_{\ell-k_0+1} \cap Q)
	=
		\mu(\widetilde{\mathcal{O}}^2_{\ell-k_0+1}(Q') \cap Q)
	\le 
	\varepsilon_0 \mu(Q), 
	\end{align*}
	where we have used that $\{\widetilde{\mathcal{O}}^2_\ell(Q')\}_{\ell=1}^{k_0}$ is a good $\varepsilon_0$-cover. All these show (c) and as result $\{\mathcal{O}^2_{\ell}\}_{\ell=1}^{2(k_0-1)+1}$ is a good $\varepsilon_0$-cover of $F$ whose length is $2(k_0-1)+1$.

	The third generation is obtained in the very same way, we take  $Q \in \F^2_{2(k_0-1)+1}$  and note that $\mu(F \cap Q)=0<\alpha \mu(Q)$. We apply again Lemma \ref{lem:cover} in $Q$ to $F \cap Q$ and obtain $\{\widetilde{\mathcal{O}}^3_{\ell}(Q)\}_{\ell=1}^{k_0}$, a good $\varepsilon_0$-cover of $F \cap Q$ (which is not empty by Remark~\ref{remark:good-cover-iter}) with $\widetilde{\F}_\ell^3(Q)\subset\D_Q$ the associated families of pairwise disjoint cubes. We set $\widetilde{\mathcal{O}}^3_\ell:=\bigcup_{Q \in \F^2_{2(k_0-1)+1}} \widetilde{\mathcal{O}}^3_\ell (Q)$ and    $\widetilde{\F}_\ell^3=\bigcup_{Q \in \F^2_{2(k_0-1)+1}} \widetilde{\mathcal{F}}^3_\ell (Q)$   for $1\le\ell\le k_0$.	
	Define $\mathcal{O}^3_\ell:=\mathcal{O}^2_\ell$ for $1\le \ell\le 2(k_0-1)+1=2k_0-1$ and $\mathcal{O}^3_\ell:=\widetilde{\mathcal{O}}^3_{\ell-2(k_0-1)}$ for $2k_0\le \ell\le  3(k_0-1)+1$. Write $\F^3_\ell\subset\D_{Q_0}$ for the associated families of pairwise disjoint dyadic cubes  for $1\le\ell\le 3(k_0-1)+1$. The same argument allows us to show that  $\{\mathcal{O}^3_{\ell}\}_{\ell=1}^{3(k_0-1)+1}$ is a good $\varepsilon_0$-cover of $F$ whose length is $3(k_0-1)+1$.
	
	If we iterate this construction in $N$ steps we will have  constructed $\{\mathcal{O}^N_{\ell}\}_{\ell=1}^{N(k_0-1)+1}$, a good $\varepsilon_0$-cover of $F$ whose length is $N(k_0-1)+1$. We observe that such iteration procedure works because $\mu(F)=0$, hence $\mu(F\cap Q)=0$ for every $Q\in\D_{Q_0}$ and also because $F\cap Q\neq\emptyset$ for every $Q$ in each of the families that define the  good $\varepsilon_0$-cover (see Remark~\ref{remark:good-cover-iter}). Since $k_0\ge 2$ and we can continue with this iteration infinitely many times we eventually obtain an infinite good $\varepsilon_0$-cover. 
\end{proof}

To continue we need to introduce some notation and some auxiliary result from \cite{CHMT}. Given $\eta=2^{-k_*}<1$ small enough to be chosen momentarily and given $Q\in\D(\pom)$ we define $Q^{(\eta)}\in\D_Q$ to be the unique dyadic cube such that $x_Q\in Q^{(\eta)}$ and $\ell(Q^{(\eta)})=\eta\ell(Q)$. 

\begin{lemma}[{\cite[Lemma~3.24]{CHMT}}]\label{lemma:oscilla}
There exist $0<\eta=2^{-k_*} \ll 1$ and $\varepsilon_0\ll 1$ small enough, and $c_0\in(0,1/2)$ (depending only on dimension, the $1$-sided CAD constants and the ellipticity of $L$)  with the following significance. Suppose that $\emptyset \neq F\subset Q_0 \in\D(\pom)$ is a Borel set and that $\{\mathcal{O}_{\ell}\}_{\ell=1}^{k}$ is a good $\varepsilon_0$-cover of $F$ with respect to $\omega_{L}^{X_{Q_0}}$ of length $k\in \N$, with associated pairwise disjoint families $\{\F_\ell\}_{1\le\ell\le k}\subset \D_{Q_0}$.  Define, $\mathcal{O}_{\ell}^{(\eta)}=\bigcup_{Q\in\F_\ell} Q^{(\eta)}$, for each $1\le \ell \le k$, and consider the Borel set $S_k:=\bigcup_{\ell=2}^k (\mathcal{O}_{\ell-1}^{(\eta)}\setminus \mathcal{O}_{\ell})$. For each $y\in F$ and $1\le \ell \le k$, let $Q^\ell(y)\in\F_\ell$ be the unique dyadic cube containing $y$, and let $P^\ell(y)\in \D_{Q^\ell(y)}$ be the unique dyadic cube containing $y$ with $\ell(P^\ell(y))=\eta\ell(Q^\ell(y))$. Then $u_k(X):=\omega_L^X(S_k)$, $X\in\Omega$, satisfies 
\begin{align}\label{eq:uS}
	|u_k(X_{(Q^{\ell}(y))^{(\eta)}}) - u_k(X_{(P^{\ell}(y))^{(\eta)}})| \ge c_0,
	\qquad  1\le \ell\le k-1.
\end{align}
\end{lemma}

\medskip

We are now ready to prove  Lemma~\ref{lem:square}:

\begin{proof}[Proof of Lemma~\ref{lem:square}]
	Fix $Q_0 \in \D$ and a Borel set $\emptyset\neq F \subset Q_0$ with $\w_L^{X_{Q_0}}(F )=0$. Let $\eta=2^{-k_*}$ and $ \varepsilon_0$ be small enough, and $c_0$ from  Lemma~\ref{lemma:oscilla}.  From Lemma~\ref{lemma:proppde} and Harnack's inequality we have that $\w_L^{X_{Q_0}}$ is Borel regular dyadically doubling measure on $Q_0$. Applying Lemma~\ref{lem:infinite} with $\mu=\w_L^{X_{Q_0}}$, one can find $\{\mathcal{O}_{\ell}\}_{\ell=1}^{\infty}$, a good $\varepsilon_0$-cover of $F$ of length $\infty$. In particular, for every $N\in \N$,  $\{\mathcal{O}_{\ell}\}_{\ell=1}^{N}$ is a good $\varepsilon_0$-cover of $F$ of length $N$. As such we can invoke Lemma~\ref{lemma:oscilla} to obtain $S_N:=\bigcup_{\ell=2}^N (\mathcal{O}_{\ell-1}^{(\eta)}\setminus \mathcal{O}_{\ell})$ so that with the notation introduced in that result the $L$-solution $u_N(X):=\omega_L^X(S_N)$, $X\in\Omega$, satisfies 
	\begin{align}\label{eq:uS:proof}
		|u_N(X_{(Q^{\ell}(y))^{(\eta)}}) - u_N(X_{(P^{\ell}(y))^{(\eta)}})| \ge c_0,
		\qquad \forall\,y\in F, \ 1\le \ell\le N-1.
	\end{align}
	
	Define next $S:=\bigcup_{\ell=2}^\infty (\mathcal{O}_{\ell-1}^{(\eta)}\setminus \mathcal{O}_{\ell})$ and $u(X):=\omega_L^X(S)$, $X\in\Omega$. By the monotone convergence theorem $u_N(X)\longrightarrow u(X)$ as $N\to\infty$ for every $X\in\Omega$. Thus, \eqref{eq:uS:proof} readily gives
		\begin{align}\label{eq:uS:proof1}
		|u(X_{(Q^{\ell}(y))^{(\eta)}}) - u(X_{(P^{\ell}(y))^{(\eta)}})| \ge c_0,
		\qquad \forall\,y\in F, \ \ell\ge 1.
	\end{align}
	This and the argument in \cite[p.~7919]{CHMT} imply
	\begin{align}\label{eq:Qil}
	\bigg(\iint_{U_{Q^{\ell}(y), \eta^3}^{\vartheta_0}} |\nabla u(Y)|^2 \delta(Y)^{1-n}\, dY\bigg)^{\frac12}
	\gtrsim_\eta c_0, \qquad	\forall\, \ell\ge 1. 
\end{align}	
Thus, for every $N\ge 1$ and every $y\in F$ 
\begin{multline*}
N c_0^2
\lesssim_\eta 
\sum_{\ell=1}^N \iint_{U_{Q^{\ell}(y), \eta^3}^{\vartheta_0}} |\nabla u(Y)|^2 \delta(Y)^{1-n}\, dY
\lesssim_\eta
\iint_{\bigcup_{\ell=1}^NU_{Q^{\ell}(y), \eta^3}^{\vartheta_0}} |\nabla u(Y)|^2 \delta(Y)^{1-n}\, dY
\\
\le
\iint_{\bigcup\limits_{y\in Q\in\D_{Q_0}}U_{Q, \eta^3}^{\vartheta_0}} |\nabla u(Y)|^2 \delta(Y)^{1-n}\, dY
=
S_{Q_0}^{\vartheta_0,\eta} u(y)^2,
\end{multline*}
where we have used that the family $\{U_{Q,\eta^3}^{\vartheta_0}\}_{Q\in\D_{\pom}}$ has bounded overlap albeit with a constant that depends on $\eta$.
Letting $N\to\infty$ we conclude that $S_{Q_0}^{\vartheta_0,\eta} u(y)=\infty$ for every $y\in F$ and the proof is complete. 
\end{proof}


\end{document}